\documentclass[11pt,a4paper,twoside]{amsart}

\usepackage{amssymb}
\usepackage{bm}
\usepackage{graphics}
\usepackage[mathcal]{euscript}
\usepackage{tikz}

\newtheorem{theo}{\bf Theorem}[section]
\newtheorem{prop}[theo]{\bf Proposition}
\newtheorem{lemma}[theo]{\bf Lemma}
\newtheorem{conj}[theo]{\bf Conjecture}
\newtheorem{defi}[theo]{\bf Definition}
\newtheorem{coro}[theo]{\bf Corollary}
\theoremstyle{plain}
\newtheorem{example}{Example}[section]
\theoremstyle{remark}

\newcommand{\N}{{\mathbb N}}

\newcommand{\C}{{\mathbb C}}
\newcommand{\calN}{{\mathcal N}}
\newcommand{\calM}{{\mathcal M}}
\newcommand{\fillings}{{\mathcal T}}

\newcommand{\RAWcolstrictfillings}[1]{#1_{\mbox{\rm\tiny col-str}}}
\newcommand{\RAWcolposfillings}[1]{#1_{\mbox{\rm\tiny col-pos}}}
\newcommand{\RAWflatcolstrictfillings}[1]{\RAWcolstrictfillings{#1}^{\mbox{\rm\tiny flat}}}
\newcommand{\colstrictfillings}{\RAWcolstrictfillings{\fillings}}
\newcommand{\colposfillings}{\RAWcolposfillings{\fillings}}
\newcommand{\flatcolstrictfillings}{{\RAWflatcolstrictfillings{\fillings}}}
\newcommand{\flatcolstrictfillingsrestr}{{\RAWflatcolstrictfillings{{\fillings'}}}}
\newcommand{\staircasecolposfillings}{\colposfillings^{\mbox{\rm\tiny staircase}}}
\newcommand{\enpermfillings}{{\mathcal R}}
\newcommand{\calS}{{\mathcal S}}
\newcommand{\hcalS}{{\hat{\mathcal S}}}
\newcommand{\bcalS}{{\bar{\mathcal S}}}
\newcommand{\bcalN}{{\bar{\mathcal N}}}

\newcommand{\hcalSrestr}{\hcalS'}
\newcommand{\bcalNrestr}{\bcalN'}
\newcommand{\hpi}{{\hat{\pi}}}
\newcommand{\bpi}{{\bar{\pi}}}
\newcommand{\bM}{{\bar{M}}}

\newcommand{\bS}{{\bar{S}}}
\newcommand{\bSsilly}{{{\bS}_{\rm silly}}}
\newcommand{\Ssilly}{{S_{\rm silly}}}
\newcommand{\bN}{{\bar{N}}}

\newcommand{\bSsillyrestr}{{{\bS'}_{\rm silly}}}
\newcommand{\bNrestr}{{\bN'}}

\newcommand{\Ssillyrestr}{{S'_{\rm silly}}}
\newcommand{\Nrestr}{N'}
\newcommand{\Irestr}{I'}

\newcommand{\block}[1]{\ensuremath{\langle #1 \rangle}}

\newcommand{\bX}{{\mathbf{X}}}
\newcommand{\bt}{{\mathbf{t}}}
\newcommand{\bx}{{\mathbf{x}}}
\newcommand{\by}{{\mathbf{y}}}
\newcommand{\bz}{{\mathbf{z}}}

\newcommand{\bu}{{\mathbf{u}}}
\newcommand{\bv}{{\mathbf{v}}}
\newcommand{\bw}{{\mathbf{w}}}

\newcommand{\bupsilon}{\boldsymbol{\upsilon}}

\newcommand{\monomtwo}[4]{{[#1^{#4};#2]_{#3}}}
\newcommand{\monomthree}[6]{{[#1^{#5}#2^{#6};#3]_{#4}}}
\newcommand{\monomfour}[8]{{[#1^{#6}#2^{#7}#3^{#8};#4]_{#5}}}

\newcommand{\formal}[2]{{#1\langle\!\langle #2\rangle\!\rangle}}
\providecommand{\inner}[2]{\langle #1,#2\rangle}

\newcommand{\barred}{X}
\newcommand{\hatted}{X}
\newcommand{\hattedplus}{X^{+}}
\newcommand{\markedadj}{X}
\newcommand{\Psilly}{P_{\rm silly}}
\newcommand{\Qsilly}{Q_{\rm silly}}
\newcommand{\psilly}{p_{\rm silly}}
\newcommand{\qsilly}{q_{\rm silly}}
\newcommand{\Psillyplus}{{\Psilly^{+}}}
\newcommand{\Qsillyplus}{{\Qsilly^{+}}}
\newcommand{\phisilly}{\phi_{\rm silly}}
\newcommand{\irr}{{\rm irr}}

\DeclareMathOperator{\rne}{rne} 
\DeclareMathOperator{\Rne}{Rne} 
\DeclareMathOperator{\lcr}{lcr} 
\DeclareMathOperator{\Lcr}{Lcr} 
\DeclareMathOperator{\rcr}{rcr} 
\DeclareMathOperator{\Rcr}{Rcr} 
\DeclareMathOperator{\Rcrproper}{{\Rcr^0}}
\DeclareMathOperator{\rcrproper}{{\rcr^0}}
\DeclareMathOperator{\Lcrproper}{{\Lcr^0}}
\DeclareMathOperator{\lcrproper}{{\lcr^0}}
\DeclareMathOperator{\LRcr}{LRcr}
\DeclareMathOperator{\lrcr}{lrcr}
\DeclareMathOperator{\Radj}{Radj}
\DeclareMathOperator{\Ladj}{Ladj}
\DeclareMathOperator{\radj}{radj}
\DeclareMathOperator{\ladj}{ladj}

\DeclareMathOperator{\Des}{Des}

\DeclareMathOperator{\Asc}{Asc}
\DeclareMathOperator{\Ascsilly}{Asc_{silly}}
\DeclareMathOperator{\adjAsc}{Asc^{short}}
\DeclareMathOperator{\Ascproper}{{\Asc^{long}}}
\DeclareMathOperator{\Ascbottom}{Asc_{bottom}}
\DeclareMathOperator{\Ascbottomproper}{{{Asc}_{bottom}^{long}}}
\DeclareMathOperator{\Rep}{Rep}
\DeclareMathOperator{\Dots}{Dots}
\DeclareMathOperator{\asc}{asc}
\DeclareMathOperator{\adjasc}{asc^{short}}
\DeclareMathOperator{\ascproper}{{\asc^{long}}}
\DeclareMathOperator{\des}{des}
\DeclareMathOperator{\rep}{rep}
\DeclareMathOperator{\comp}{comp}
\DeclareMathOperator{\boxcomp}{boxcomp}
\DeclareMathOperator{\bcomp}{\overline{comp}}
\DeclareMathOperator{\lmin}{lmin}
\DeclareMathOperator{\rmin}{rmin}
\DeclareMathOperator{\Rmin}{Rmin}
\DeclareMathOperator{\lmax}{lmax}
\DeclareMathOperator{\rmax}{rmax}
\DeclareMathOperator{\Min} {Min}
\DeclareMathOperator{\Max} {Max}
\DeclareMathOperator{\Rmax}{Rmax}

\DeclareMathOperator{\flatten}{Fl}
\DeclareMathOperator{\steepen}{St}
\DeclareMathOperator{\lev}{lev}   
\DeclareMathOperator{\inter}{int} 
\DeclareMathOperator{\opener}{op}     
\DeclareMathOperator{\closer}{cl}     

\providecommand{\abs}[1]{\lvert#1\rvert}

\newcommand{\PATTERNPsillyplain}{
    \draw[xshift=14pt, yshift=14pt, thick] (0,2) -- (3,2);
    \draw[xshift=14pt, yshift=14pt, thick] (2,0) -- (2,3);
    \foreach \x/\y in {1/3,2/1,3/2} \draw[fill=black] (\x,\y) circle (7pt);
}

\newcommand{\PATTERNQsillyplain}{
    \draw[xshift=14pt, yshift=14pt, thick] (0,2) -- (3,2);
    \draw[xshift=14pt, yshift=14pt, thick] (1,0) -- (1,3);
    \foreach \x/\y in {1/1,2/2,3/3} \draw[fill=black] (\x,\y) circle (7pt);
}

\newcommand{\PATTERNPsilly}{
    \draw[xshift=14pt, yshift=14pt, thick] (0,2) -- (3,2);
    \draw[xshift=14pt, yshift=14pt, thick] (2,0) -- (2,3);
    \foreach \x/\y/\r in {1/3/6,2/1/6,3/2/8} \draw[fill=black] (\x,\y) circle (\r pt);
}

\newcommand{\PATTERNQsilly}{
    \draw[xshift=14pt, yshift=14pt, thick] (0,2) -- (3,2);
    \draw[xshift=14pt, yshift=14pt, thick] (1,0) -- (1,3);
    \foreach \x/\y/\r in {1/1/6,2/2/8,3/3/6}
    \draw[fill=black] (\x,\y) circle (\r pt);
}

\newcommand{\PATTERNP}{
    \draw[xshift=14pt, yshift=14pt, thick] (0,2) -- (3,2);
    \draw[xshift=14pt, yshift=14pt, thick] (1,0) -- (1,3);
    \foreach \x/\y/\r in {1/1/6,2/3/8,3/2/6} \draw[fill=black] (\x,\y) circle (\r pt);
}

\newcommand{\PATTERNQ}{
    \draw[xshift=14pt, yshift=14pt, thick] (0,2) -- (3,2);
    \draw[xshift=14pt, yshift=14pt, thick] (2,0) -- (2,3);
    \foreach \x/\y/\r in {1/2/6,2/1/6,3/3/8} \draw[fill=black] (\x,\y) circle (\r pt);
}

\newcommand{\PATTERNR}{
    \draw[xshift=14pt, yshift=14pt, thick] (0,1) -- (2,1);
    \draw[xshift=14pt, yshift=14pt, thick] (1,0) -- (1,2);
    \foreach \x/\y/\r in {1/1/6,2/2/8} \draw[fill=black] (\x,\y) circle (\r pt);
}

\newcommand{\PATTERNascent}{
    \draw[xshift=14pt, yshift=14pt, thick] (1,0) -- (1,2);
    \foreach \x/\y/\r in {1/1/6,2/2/8} \draw[fill=black] (\x,\y) circle (\r pt);
}

\newcommand{\PATTERNlascent}{
    \draw[xshift=14pt, yshift=14pt, thick, dotted] (0,1) -- (2,1);
    \draw[xshift=14pt, yshift=14pt, thick] (1,0) -- (1,2);
    \foreach \x/\y/\r in {1/1/6,2/2/8} \draw[fill=black] (\x,\y) circle (\r pt);
}

\newcommand{\patternPsillyplain}{\!\raisebox{-0.5em}{
  \begin{tikzpicture}[line width=0.7pt, scale=0.15]
    \tikzstyle{disc} = [circle,thin,draw=black, minimum size=1.7pt, inner sep=0pt ]
    \PATTERNPsillyplain
  \end{tikzpicture}}
}

\newcommand{\patternQsillyplain}{\!\raisebox{-0.5em}{
  \begin{tikzpicture}[line width=0.7pt, scale=0.15]
    \tikzstyle{disc} = [circle,thin,draw=black, minimum size=1.7pt, inner sep=0pt ]
    \PATTERNQsillyplain
  \end{tikzpicture}}
}

\newcommand{\patternPsilly}{\!\raisebox{-0.4em}{
  \begin{tikzpicture}[line width=0.7pt, scale=0.15]
    \tikzstyle{disc} = [circle,thin,draw=black, minimum size=1.7pt, inner sep=0pt ]
    \PATTERNPsilly
  \end{tikzpicture}}
}

\newcommand{\patternQsilly}{\!\raisebox{-0.4em}{
  \begin{tikzpicture}[line width=0.7pt, scale=0.15]
    \tikzstyle{disc} = [circle,thin,draw=black, minimum size=1.7pt, inner sep=0pt ]
    \PATTERNQsilly
  \end{tikzpicture}}
}

\newcommand{\patternP}{\!\raisebox{-0.4em}{
  \begin{tikzpicture}[line width=0.7pt, scale=0.15]
    \tikzstyle{disc} = [circle,thin,draw=black, minimum size=1.7pt, inner sep=0pt ]
    \PATTERNP
  \end{tikzpicture}}
}

\newcommand{\patternQ}{\!\raisebox{-0.4em}{
  \begin{tikzpicture}[line width=0.7pt, scale=0.15]
    \tikzstyle{disc} = [circle,thin,draw=black, minimum size=1.7pt, inner sep=0pt ]
    \PATTERNQ
  \end{tikzpicture}}
}

\newcommand{\patternR}{\!\raisebox{-0.2em}{
  \begin{tikzpicture}[line width=0.7pt, scale=0.15]
    \tikzstyle{disc} = [circle,thin,draw=black, minimum size=1.7pt, inner sep=0pt ]
    \PATTERNR
  \end{tikzpicture}}
}

\newcommand{\patternascent}{\!\raisebox{-0.2em}{
  \begin{tikzpicture}[line width=0.7pt, scale=0.15]
    \tikzstyle{disc} = [circle,thin,draw=black, minimum size=1.7pt, inner sep=0pt ]
    \PATTERNascent
  \end{tikzpicture}}
}

\newcommand{\patternlascent}{\!\raisebox{-0.2em}{
  \begin{tikzpicture}[line width=0.7pt, scale=0.15]
    \tikzstyle{disc} = [circle,thin,draw=black, minimum size=1.7pt, inner sep=0pt ]
    \PATTERNlascent
  \end{tikzpicture}}
}

\begin{document}

\title{Equidistributed statistics on matchings and permutations}
\author{Niklas Eriksen \and Jonas Sj{\"o}strand}
\address{School of Science and Technology, {\"O}rebro University
  SE-701 82 {\"O}rebro, Sweden} 
\email{niklas.eriksen@oru.se}
\address{Department of Mathematics, Royal Institute of Technology \\
  SE-100 44 Stockholm, Sweden}
\email{jonass@kth.se}

\date{November 11, 2011}

\begin{abstract}
We show
that the bistatistic of right nestings and right crossings in
matchings without left nestings is equidistributed with the number of
occurrences of two certain
patterns in permutations, and
furthermore that this equidistribution holds when refined to positions
of these statistics in matchings and permutations.
For this distribution we obtain a non-commutative
generating function which specializes to Zagier's generating function
for the Fishburn numbers after abelianization.

As a special case we obtain proofs of two conjectures of Claesson and Linusson.

Finally, we conjecture that our results can be generalized to
involving left crossings of matchings too.
\end{abstract}

\maketitle

\section{Introduction}
\noindent
Permutations of $n$ elements are counted by $n!$. To a
combinatorialist, objects counted by some number are often regarded as
a representation or concretisation of that number. For each
representation we obtain of a number sequence, we understand it
better. The more objects we have that are counted by $n!$, the better
we will understand that sequence of numbers. 

In 2010, Bousquet-M\'elou et al.~\cite{BouClaDukKit2010} showed
bijectively that the Fishburn numbers, starting with $1, 1, 2, 5, 15,
53, 217\dots$, count not only the unlabelled $(\mathrm{\bf 2} +
\mathrm{\bf 2})$-free posets on $n$ elements which define them, but
also ascent sequences of length $n$, permutations of $[n]$ avoiding
the pattern \patternPsillyplain, and matchings on $[2n]$ with no
neighbour nestings. This variety of representations of these numbers
indicate that they are well worth studying. 

Neighbour nestings in matchings come in two flavours, right and left
nesting. As it turns out, the number of matchings avoiding only one of
these flavours, say left nestings, are counted by $n!$. The same
holds, obviously, for permutations on $[n]$, and also for
$(\mathrm{\bf 2} + \mathrm{\bf 2})$-free posets with some natural
labelling, according to a 2011 paper by Claesson and
Linusson \cite{ClaLin2011}.

In more detail, we may view a matching on $[2n]$ as a set of arcs between
the elements of $[2n]$ laid out on the real line. A nesting is a pair of
arcs, one enclosing the other, and a right (left) nesting is a nesting
where the two right (left) elements are adjacent. In the matching
below, there is a right nesting between edges $\{2, 7\}$ and $\{4,
6\}$, but no other left or right nestings:

\begin{center}
\begin{tikzpicture}[line width=0.7pt, scale=0.4]
  \tikzstyle{disc} = [circle,thin,draw=black, minimum size=1.7pt,
    inner sep=0pt ]
  \draw (0.4, 0) -- (8.6, 0);
  \foreach \x in {1, ..., 8} \draw[fill=black] (\x, 0) circle (3pt);
  \foreach \x in {1, ..., 8} \draw (\x, -1) node{\scriptsize{\x}};
  \foreach \c/\r in {3/1,8/1.5} 
    \draw (\c, 0) arc (0:180:\r cm);
  \foreach \c/\r in {6/1,7/2.5} 
    \draw[line width=1.2pt] (\c, 0) arc (0:180:\r cm);
\end{tikzpicture}
\end{center}

Turning to permutations, an instance of the pattern
\patternPsillyplain in a permutation $\pi$ is a triple of entries
with the same positions and values, relative to each other, where
vertical lines indicate adjacency in position and horizontal lines 
indicate adjacency in value. Thus, the permutation 3412 contains
the pattern \patternPsillyplain once, namely {\bf 3}4{\bf 12}. (Note
that while 3{\bf 412} is a 312 pattern and the last two entries in
the pattern are adjacent in position, the first and last entries in
the pattern are not adjacent in value and thus it is not an instance
of the pattern \patternPsillyplain.) 

To prove that permutations and matchings without left nestings are
equinumerous, Claesson and Linusson took the following
approach. Consider an inversion table, that is an $n$-tuple $\alpha =
(\alpha_1, \dots, \alpha_n)$ such that 
$1 \leq \alpha_j \leq j$. To each inversion table $\alpha$, we obtain a
permutation if we traverse $\alpha$ from left to right, inserting the
number $j$ in position $\alpha_j$ from the left. Similarly, we obtain
a matching if we traverse $\alpha$ from left to right, inserting the
right end of a new arc to the far right and the left end immediately
to the left of the right end of the $\alpha_j$th right end from the
left. These two bijections induce a bijection between permutations of
$[n]$ and matchings of $[2n]$ with no left nestings, since we
would need to insert the left end of a new arc to the left of a
previous left end to obtain a left nesting. Note also that this
bijection ties the given examples of a matching and a permutation
together, via the inversion table $\alpha = (1, 2, 1, 2)$. 

During their work, Claesson and Linusson made the observation that the
number of left-nesting free matchings of $[2n]$ with $k$ right
nestings seemed to coincide with the number of permutations with $k$
occurrences of the pattern \patternPsillyplain, which would make an
interesting refinement. Indeed, their simple recursive bijection gives
a right nesting in step $j$ if and 
only if $\alpha_j < \alpha_{j-1}$ and the permutation in step $j$
contains the pattern \patternPsillyplain with top entry
$j$ if and only if $\alpha_j < \alpha_{j-1}$. Case closed, right?

Not really! As the bijection progresses, an arc added to the matching
may break previous right nestings. Similarly, the pattern
\patternPsillyplain may also be broken when more entries are included
in the permutation. For example, putting the left end of a new arc
between the arc with the right nesting in the matching drawn above
will break the nesting: 
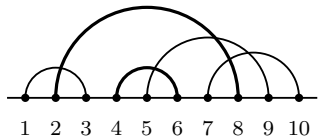
\begin{center}
\begin{tikzpicture}[line width=0.7pt, scale=0.4]
  \tikzstyle{disc} = [circle,thin,draw=black, minimum size=1.7pt,
    inner sep=0pt ]
  \draw (0.4, 0) -- (10.6, 0);
  \foreach \x in {1, ..., 10} \draw[fill=black] (\x, 0) circle (3pt);
  \foreach \x in {1, ..., 10} \draw (\x, -1) node{\scriptsize{\x}};
  \foreach \c/\r in {3/1,9/2,10/1.5} 
    \draw (\c, 0) arc (0:180:\r cm);
  \foreach \c/\r in {6/1,8/3} 
    \draw[line width=1.2pt] (\c, 0) arc (0:180:\r cm);
\end{tikzpicture}
\end{center}
In the corresponding permutation, via inversion table $\alpha = (1, 2,
1, 2, 3)$, we should add a 5 in position 3, 
obtaining 34512. This permutation retains its occurrence of the pattern
\patternPsillyplain. On the other hand, inserting the 5 in position 4
would have destroyed the pattern and the corresponding arc added to
the matching would not have ruined the nesting. It is not hard to see
that the total number of right nestings in all matchings of $[2n]$
equals the total number of occurrences of \patternPsillyplain in
permutations of length $n$, but their equidistribution is not clear
from the bijection.

This paper presents a proof of the bijection. We actually show
that the bistatistic of right nestings and right crossings
(pairs of arcs with neighbouring right ends that cross each other) in
matchings is equidistributed with the number of occurrences of the
patterns \patternPsillyplain and \patternQsillyplain in permutations, and
furthermore that this equidistribution holds when refined to
positions/values of these statistics in matchings and permutations. 

Our proof is built on the bijection from Claesson and Linusson, but
the inversion table in their proof is replaced by a
richer structure in ours, namely certain fillings of partition shapes.
(We were inspired by Krattenthaler~\cite{kra2006}.)

Given a set $X\subseteq\{2,3,\dotsc,n\}$ associated with a partition shape,
we obtain a bijection between some matchings of $[2n]$ and some
permutations of $[n]$ such that the occurrences of right adjacencies
and patterns in positions in $X$ match perfectly, whereas other
occurrences of right adjacencies and patterns may not match at all. We
then apply the sieve principle to obtain a perfect match in the case
$X = \emptyset$.

The bijections will be presented in more detail in the following
four sections, where Section \ref{sc: conjecture 20} ties our first
results together into a proof of Conjecture 20 in \cite{ClaLin2011}. 
We then introduce noncommutative formal power series in
Section \ref{sc: NFPS}. Section \ref{sc: generating functions} gives
generating functions of matchings and permutations as noncommutative
formal power series and we prove that these two generating functions
are identical, giving the equidistribution when refined to
positions. We then abelianize the generating functions and demonstrate
that they coincide with previously known generating functions in
special cases. At the end of Section~\ref{sc: abel} we obtain a proof of
Conjecture~21 in~\cite{ClaLin2011}. Finally, in Section~\ref{sc: open}
we conclude with some results and conjectures
concerning the number of left crossings.

\section{Shapes and fillings}
\noindent
For a nonnegative integer $n\in\N$ we let $[n]$ denote the interval
$\{1,2,\dotsc,n\}$; in particular, $[0]:=\emptyset$.

\subsection{Shapes}
\noindent
A \emph{partition}
is a weakly increasing sequence of positive integers
$\lambda=(\lambda_1\le\lambda_2\le\dotsb\le\lambda_{\ell})$,
and we will identify
$\lambda$ with its \emph{shape}
which is the bottom-justified arrangement of squares, called \emph{cells},
with $\lambda_i$ cells
in the $i$th column from the left. The left part of Figure~\ref{fig: staircase}
shows an example. (This orientation of the shape is more suitable
for us than the usual French or English Young diagram.)

Define the \emph{lazy set} of $\lambda$, denoted by $X(\lambda)$,
to be the set of indices $i$ such that $\lambda_{i-1}=\lambda_{i}$.
Also, let $\ell(\lambda)=\ell$ denote the
\emph{length} (i.e.~the number of columns)
of $\lambda$.
A shape is said to be \emph{flat} if its
rows have distinct lengths and \emph{steep} if its columns have distinct
lengths. If it is both flat and steep it is a \emph{staircase} shape.
(Note that a flat shape is completely determined by its length
and its lazy set.)

We will number the columns of a shape from left to right and the rows
from bottom to top, and the cells will be identified with their coordinate
pairs $(i,j)$ where $i$ is the column index and $j$ is the row index.

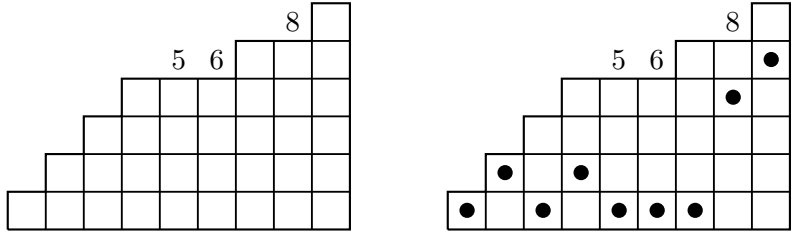
\begin{figure}
\centering
\begin{tikzpicture}[line width=0.7pt, scale=0.5]
  \tikzstyle{disc} = [circle,thin,draw=black, minimum size=1.7pt,
    inner sep=0pt ]
  \foreach \x/\y in {1/1,2/2,3/3,6/4,7/5} \draw[xshift=14pt,
    yshift=14pt ] (\x,\y) -- (9,\y);
  \foreach \x/\y in {1/1,2/2,3/3,4/4,5/4,6/4,7/5,8/5} \draw[xshift=14pt,
    yshift=14pt ] (\x,\y) -- (\x, 0); 
  \draw[xshift=14pt, yshift=14pt, thick] (0, 0) -- (9, 0) --
  (9, 6) -- (8, 6) -- (8, 5) -- (6, 5) -- (6, 4) -- (3, 4) --
  (3, 3) -- (2, 3) -- (2, 2)
  -- (1, 2) -- (1, 1) -- (0, 1) -- (0, 0);
  \foreach \x/\y in {5/4,6/4,8/5} \draw[xshift=14pt,
    yshift=14pt] (\x cm-0.5cm,\y cm+0.5cm) node{\x}; 
\end{tikzpicture}
\hspace{1cm}
\begin{tikzpicture}[line width=0.7pt, scale=0.5]
  \tikzstyle{disc} = [circle,thin,draw=black, minimum size=1.7pt,
    inner sep=0pt ]
  \foreach \x/\y in {1/1,2/2,3/3,6/4,7/5} \draw[xshift=14pt,
    yshift=14pt ] (\x,\y) -- (9,\y);
  \foreach \x/\y in {1/1,2/2,3/3,4/4,5/4,6/4,7/5,8/5} \draw[xshift=14pt,
    yshift=14pt ] (\x,\y) -- (\x, 0); 
  \draw[xshift=14pt, yshift=14pt, thick] (0, 0) -- (9, 0) --
  (9, 6) -- (8, 6) -- (8, 5) -- (6, 5) -- (6, 4) -- (3, 4) --
  (3, 3) -- (2, 3) -- (2, 2)
  -- (1, 2) -- (1, 1) -- (0, 1) -- (0, 0);
  \foreach \x/\y in {5/4,6/4,8/5} \draw[xshift=14pt,
    yshift=14pt] (\x cm-0.5cm,\y cm+0.5cm) node{\x}; 
  \foreach \x/\y in {1/1,2/2,3/1,4/2,5/1,6/1,7/1,8/4,9/5}
  \draw[fill=black] (\x,\y) circle (5pt);
\end{tikzpicture}
\caption{Left: The shape $\lambda = (1, 2, 3, 4,
  4, 4, 5, 5, 6)$ with $\ell(\lambda) = 9$ and lazy set $X(\lambda) = \{5,
  6, 8\}$. Right: The filling $T$ with
  $\lambda(T) = \lambda$ as before and dots $\alpha(T) = (1, 2, 1, 2,
  1, 1, 1, 4, 5)$.
} \label{fig: staircase} 
\end{figure}

\subsection{Fillings}
\noindent
A \emph{0-1-filling}, or simply a \emph{filling},
of a shape is an assignment of the number 0
or 1 to each cell of the shape.
If each column sum is positive
or equal to one, the filling is said to be \emph{column-positive} or
\emph{column-strict}, respectively. The attributes
\emph{row-positive} and \emph{row-strict} are defined analogously,
and a filling that is both row- and column-positive is simply called
\emph{positive} and if it is row- and column-strict it is called
\emph{strict}.

The total sum (i.e.~the number of ones) of a filling $T$ is denoted by $n(T)$.
We let $\lambda(T)$ denote the underlying shape of $T$
and for convenience we write $\ell(T)$ and $X(T)$ as
shorthand for $\ell(\lambda(T))$ and $X(\lambda(T))$.
A filling also inherits the properties ``flat'' and ``steep''
from its underlying shape.

Let $\fillings(\lambda)$ denote the set of fillings of the shape $\lambda$,
and let $\fillings=\bigcup_{\lambda}\fillings(\lambda)$ be the set of
all fillings (of nonempty shapes). Also, if $p$ is a filling property we let
$\fillings_p(\lambda)$ denote the set of fillings of $\lambda$ with
property $p$, and if $q$ is a shape property we let
$\fillings_p^q$ denote the set of fillings with property $p$ of
shapes with property $q$. For instance, we write
$\flatcolstrictfillings$ for the set of flat column-strict fillings.

When drawing fillings we will omit the zeroes and
replace the ones by dots, and to each filling $T$ we associate the set
of dots
$\Dots(T)=
\{(i,j)\colon\ \text{there is a dot at}\ (i,j)\}$
and a set sequence
$\alpha(T) = (\alpha_1, \dots, \alpha_{\ell(T)})$ where
$\alpha_i=\{j\colon (i,j)\in\Dots(T)\}$
contains the row indices of the dots in column $i$.
If $T$ is column-strict all sets $\alpha_i$ are singletons
and, for convenience, in this case we will abuse notation and
make no distinction between $\alpha_i=\{j\}$ and $j$ itself.
(Note that if $\lambda$ is the staircase shape of length $n$ then
$\colstrictfillings(\lambda)$ can be interpreted as the set
of inversion tables of size $n$.)

Given two fillings $T$ and $T'$, we define their
\emph{direct sum} $T\oplus T'$ to be the filling obtained by
putting $T'$ to the north-east of $T$ and filling out
the rectangle below $T'$ with empty cells. Figure~\ref{fig:fillingdirectsum}
shows an example.
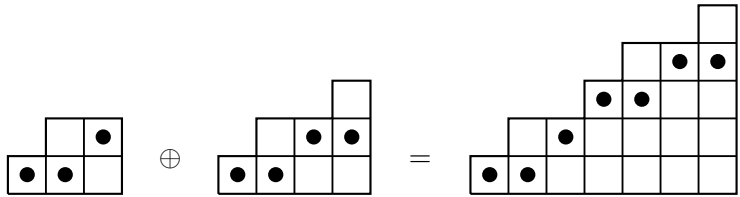
\begin{figure}
\centering
$
\raisebox{-1em}{
\begin{tikzpicture}[line width=0.7pt, scale=0.5]
  \tikzstyle{disc} = [circle,thin,draw=black, minimum size=1.7pt,
    inner sep=0pt ]
  \foreach \x/\y in {1/1} \draw[xshift=14pt,
    yshift=14pt ] (\x,\y) -- (3,\y);
  \foreach \x/\y in {1/1,2/2} \draw[xshift=14pt,
    yshift=14pt ] (\x,\y) -- (\x, 0); 
  \draw[xshift=14pt, yshift=14pt, thick] (0, 0) -- (3, 0) --
  (3, 2) -- (1, 2) -- (1, 1) -- (0, 1) -- (0, 0);
  \foreach \x/\y in {1/1,2/1,3/2}
  \draw[fill=black] (\x,\y) circle (5pt);
\end{tikzpicture}
}
\ \ \oplus\ \ 
\raisebox{-1em}{
\begin{tikzpicture}[line width=0.7pt, scale=0.5]
  \tikzstyle{disc} = [circle,thin,draw=black, minimum size=1.7pt,
    inner sep=0pt ]
  \foreach \x/\y in {1/1,3/2} \draw[xshift=14pt,
    yshift=14pt ] (\x,\y) -- (4,\y);
  \foreach \x/\y in {1/1,2/2,3/2} \draw[xshift=14pt,
    yshift=14pt ] (\x,\y) -- (\x, 0); 
  \draw[xshift=14pt, yshift=14pt, thick] (0, 0) -- (4, 0) --
  (4, 3) -- (3, 3) -- (3, 2) -- (1, 2) -- (1, 1) -- (0, 1) -- (0, 0);
  \foreach \x/\y in {1/1,2/1,3/2,4/2}
  \draw[fill=black] (\x,\y) circle (5pt);
\end{tikzpicture}
}
\ \ =\ \ 
\raisebox{-1em}{
\begin{tikzpicture}[line width=0.7pt, scale=0.5]
  \tikzstyle{disc} = [circle,thin,draw=black, minimum size=1.7pt,
    inner sep=0pt ]
  \foreach \x/\y in {1/1,3/2,4/3,6/4} \draw[xshift=14pt,
    yshift=14pt ] (\x,\y) -- (7,\y);
  \foreach \x/\y in {1/1,2/2,3/2,4/3,5/4,6/4} \draw[xshift=14pt,
    yshift=14pt ] (\x,\y) -- (\x, 0); 
  \draw[xshift=14pt, yshift=14pt, thick] (0, 0) -- (7, 0) --
  (7, 5) -- (6, 5) -- (6, 4) -- (4, 4) -- (4, 3) -- (3, 3) --
  (3, 2) -- (1, 2) -- (1, 1) -- (0, 1) -- (0, 0);
  \foreach \x/\y in {1/1,2/1,3/2,4/3,5/3,6/4,7/4}
  \draw[fill=black] (\x,\y) circle (5pt);
\end{tikzpicture}
}
$
\caption{Direct sum of fillings.} \label{fig:fillingdirectsum} 
\end{figure}
A filling that cannot be written as a direct sum of
two fillings are called \emph{irreducible}.
Every filling can be written uniquely as a direct sum of
irreducible fillings, called its \emph{irreducible components},
and we let $\comp(T)$ denote the number of irreducible components of $T$.

A row-strict filling $T$ may be
\emph{steepened} by, for each $k$, merging
all columns of length $k$ into one single column (with dots
in the same rows), see Figure~\ref{fig:steepen}.
Analogously, a column-strict
filling $T$ may be \emph{flattened} by merging
rows of equal length. The resulting fillings are
denoted by $\steepen(T)$ and $\flatten(T)$, respectively,
and obviously $\steepen(T)$ is steep and $\flatten(T)$ is flat.

\begin{figure}
\centering
\begin{tikzpicture}[line width=0.7pt, scale=0.5]
  \tikzstyle{disc} = [circle,thin,draw=black, minimum size=1.7pt,
    inner sep=0pt ]
  \foreach \x/\y in {2/1,3/2,3/3,5/4,5/5} \draw[xshift=14pt,
    yshift=14pt ] (\x,\y) -- (6,\y);
  \foreach \x/\y in {1/1,2/1,3/2,4/4,5/4} \draw[xshift=14pt,
    yshift=14pt ] (\x,\y) -- (\x, 0); 
  \draw[xshift=14pt, yshift=14pt, thick] (0, 0) -- (6, 0) --
  (6, 6) -- (5, 6) -- (5, 4) -- (3, 4) -- (3, 2) -- (2, 2) --
  (2, 1) -- (0, 1) -- (0, 0);
  \foreach \x/\y in {2/1,4/2,5/4,6/3,6/5,6/6}
  \draw[fill=black] (\x,\y) circle (5pt);
\end{tikzpicture}
\hspace{1cm}
$\xrightarrow{\mbox{\footnotesize steepen}}$
\hspace{1cm}
\begin{tikzpicture}[line width=0.7pt, scale=0.5]
  \tikzstyle{disc} = [circle,thin,draw=black, minimum size=1.7pt,
    inner sep=0pt ]
  \foreach \x/\y in {1/1,2/2,2/3,3/4,3/5} \draw[xshift=14pt,
    yshift=14pt ] (\x,\y) -- (4,\y);
  \foreach \x/\y in {1/1,2/2,3/4} \draw[xshift=14pt,
    yshift=14pt ] (\x,\y) -- (\x, 0); 
  \draw[xshift=14pt, yshift=14pt, thick] (0, 0) -- (4, 0) --
  (4, 6) -- (3, 6) -- (3, 4) -- (2, 4) -- (2, 2) -- (1, 2) --
  (1, 1) -- (0, 1) -- (0, 0);
  \foreach \x/\y in {1/1,3/2,3/4,4/3,4/5,4/6}
  \draw[fill=black] (\x,\y) circle (5pt);
\end{tikzpicture}
\caption{A row-strict filling to the left and the steepening of it to the
right.} \label{fig:steepen} 
\end{figure}
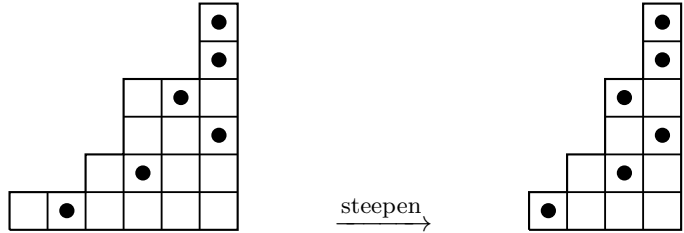

For a filling $T$, let $\Min(T)=\{i\colon (i,1)\in\Dots(T)\}$
denote the set of columns with a dot on the bottom row,
and let $\Max(T)=\{i\colon (i,\lambda(T)_i)\in\Dots(T)\}$
denote the set of columns
with a dot in their topmost cell. The cardinalities of these sets are denoted
by $\min(T)$ and $\max(T)$. Also, let
\[
\rmax(T)=\#\{(i,j)\in\Dots(T)\colon \lambda_i-j<\lambda_{i'}-j'
\ \text{for any}\ (i',j')\in\Dots(T)\ \text{with}\ i'>i\}
\]
denote the number of dots that are strictly closer to the top of its
column than is any dot strictly to the right; and
let
\[
\lmin(T)=\#\{(i,j)\in\Dots(T)\colon i<i'
\ \text{for any}\ (i',j')\in\Dots(T)\ \text{with}\ j'<j\}
\]
denote the number of dots that are strictly to the left of any
dot strictly below.

For column-strict fillings $T$
we extend our notation and let $\Rmax(T):= \{i \colon \lambda_i - \alpha_i
< \lambda_j - \alpha_j,\ \forall j>i\}$ denote the set of columns
whose dots are closer to the top than is any dot to the right.
Note that this is compatible with our previous notation in the sense that
$\#\Rmax(T)=\rmax(T)$.

Also, define the sets
\begin{align*}
\Des(T) &:= \{i \colon \alpha_{i-1} > \alpha_i\} \\
\Asc(T) &:= \{i \colon \alpha_{i-1} < \alpha_i\} \\
\Rep(T) &:= \{i \colon \alpha_{i-1} = \alpha_i\}
\end{align*}
of descents, ascents, and repetitions of a column-strict filling $T$.
The cardinalities of
these sets are denoted by $\des(T)$, $\asc(T)$, and $\rep(T)$.

\begin{example}
We will use the column-strict filling $T$ with shape
$\lambda = \lambda(T) = (1, 2, 3, 4, 4, 4, 5, 5, 6)$ and dot positions
$\alpha = \alpha(T) = 
(1, 2, 1, 2, 1, 1, 1, 4, 5)$ as a running example. In Figure \ref{fig:
staircase}, we give a graphical representation of $T$.  We have
minimal dots $\Min(T) = \{1, 3, 5, 6, 7\}$ and maximal dots $\Max(T) =
\{1, 2\}$, as well as $\Rmax(T) = \{2, 9\}$, $\Des(T) = \{3,
5\}$, $\Asc(T) = \{2, 4, 8, 9\}$, and $\Rep(T) = \{6, 7\}$. 
\end{example}

A pair of dots is said to be \emph{descending} if
one dot is strictly to the south-east of the other one,
and \emph{ascending} if one dot is strictly to the north-east of the other one.
A sequence of dots in which each dot is strictly to the
south-east (north-east) of the previous dot is called a descending
(ascending) chain.

\section{Permutations and bivincular patterns}
\noindent
Let $\calS_n$ be the set of permutations $\pi$ of $[n]=\{1,2,\dotsc,n\}$,
let $\calS=\bigcup_{n \ge 1} \calS_n$ denote the set of all
permutations, and let $n(\pi)$ denote the size of a permutation
$\pi$.

We will write a permutation $\pi\in\calS_n$
in \emph{single-row notation}, that is, as
the sequence $\pi(1)\dotsm\pi(n)$ of values. Let
$\Rmin(\pi) = \{\pi(j)\in[n] : j < k \Rightarrow \pi(j) < \pi(k)\}$ denote the
set of right minima of $\pi$ and let $\Rmax(\pi) = \{\pi(j)\in[n] : j < k
\Rightarrow \pi(j) > \pi(k)\}$ denote the set of right maxima of $\pi$.

Given a permutation $\pi$ of $[n]$ and a permutation $\pi'$ of $[n']$,
we define their \emph{direct sum} $\pi\oplus\pi'$ as the permutation
obtained by juxtaposition of (the single-row notation of)
$\pi$ and $\pi'$ whereafter the entries of
$\pi'$ are increased by $n(\pi)$. Example: $213\oplus2413=2135746$.
Every permutation $\pi$ can be written uniquely as a
direct sum of irreducible components, the number of which we
denote by $\comp(\pi)$.
Example: $\comp(2135746)=\comp(21\oplus 1\oplus 2413)=3$.

Bivincular patterns were formally defined in
\cite{BouClaDukKit2010}. These are classical patterns in permutations
drawn in permutation diagrams with optional horizontal and vertical
bars, the horizontal bars indicating adjacency in value and the
vertical bars indicating adjacency in position. Since we are interested not
only in the \emph{number} of pattern occurrences but also in their
\emph{positions} we will draw one of the dots in a pattern
slightly bigger to signify that an occurrence of the
pattern is identified with the \emph{value} of that dot.

In particular, an occurrence of the pattern \patternascent in a permutation
$\pi$ is an entry $\pi(j)$ in the single-row
notation of $\pi$ such that $\pi(j-1)<\pi(j)$. Such an occurrence is
called an \emph{ascent top}, or simply an \emph{ascent},
and the set of ascents is denoted by $\Asc(\pi)$.
We will also need the set of \emph{ascent bottoms}
$\Ascbottom(\pi) = \{\pi(j) :
\pi(j) < \pi(j+1)\}$ towards the end of this article.

Let us refine the ascent statistic. An occurence of the
pattern \patternR is an ascent $\pi(j)$ such that
$\pi(j-1)=\pi(j)-1$. These are known as \emph{adjacencies},
especially in the context of genome rearrangements
\cite{Burstein2010,Eriksen2008b}. To avoid overloading with
adjacencies in matchings, we call the pattern a \emph{short ascent}. The
\emph{long ascents}, that is where $\pi(j-1) < \pi(j) - 1$, are
occurrences of the pattern \patternlascent, where the dashed line
indicate a lack of adjacency in value.

We refine the long ascents further. Namely,
an occurrence of the pattern \patternP is an ascent $\pi(j)$
such that the entry $\pi(j)-1$ is to the right of $\pi(j)$ in the
single-row notation of $\pi$, and
an occurrence of the pattern \patternQ
is an ascent $\pi(j)$ such that $\pi(j)-1$ is to
the left of $\pi(j-1)$. Let $P(\pi)$, $Q(\pi)$, $\Ascproper(\pi)$ and
$\adjAsc(\pi)$ denote the set of occurrences of \patternP, \patternQ,
\patternlascent and \patternR in $\pi$, 
respectively. Clearly, the three sets \patternP, \patternQ and
\patternR are disjoint and their union is $\Asc(\pi)$.

\subsection{Barred permutations}
\noindent
\begin{defi}
A \emph{barred permutation} $\bpi$ is a permutation
where zero or more ascents are marked by a bar.
The set of barred elements is denoted by
$\barred(\bpi)\subseteq\Asc(\bpi)$.

The set of barred permutations is denoted by $\bcalS$.
\end{defi}

The \emph{direct sum} $\bpi\oplus\bpi'$ of two barred permutations
is defined exactly as for non-barred permutations, with the additional rule
that the barred entries of $\bpi$ and $\bpi'$ keep
their bars in $\bpi\oplus\bpi'$.
Every barred permutation $\bpi$ can be written uniquely as a
direct sum of irreducible barred components, the number of which we
denote by $\bcomp(\bpi)$. Note that $\bcomp(\bpi)$ may be smaller than
$\comp(\bpi)$, the number of irreducible components of $\bpi$ disregarding
the bars. Example: $\bcomp(21\bar{3}574\bar{6})=
\bcomp(21\bar{3}\oplus 241\bar{3})=2$ while
$\comp(2135746)=3$.

We now define a map $\phi: \flatcolstrictfillings \rightarrow\bcalS$
from flat column-strict fillings to barred permutations.
For a flat column-strict filling $T$, this
map iterates through the sequence $\alpha = \alpha(T)$ to create a
barred permutation $\bpi$.

Start with the single block $\block{1}$. Now, for $i = 2,3,\dotsc,n(T)$,
if $i \in X(T)$, insert $\bar{i}$ at the right end
of the $\alpha_i$th block from the left. If $i\notin X(T)$, insert the
new block $\block{i}$ immediately before the $\alpha_i$th block from
the left or to the right of the rightmost block if $\alpha_i=\lambda(T)_i$.
Having reached $n(T)$, we dissolve the block structure to
obtain (the single-row notation of) $\bpi$.

\begin{example}
Using $\phi$, let us map the flat column-strict filling $T$ with shape
$\lambda(T) = (1, 2, 3, 4, 4, 4,
5, 5, 6)$ and $\alpha = \alpha(T) = (1, 2, 1, 2, 1, 1, 1, 4, 5)$ into
$\bcalS$. Recall that the lazy set of $\lambda(T)$ is
$X(T) = \{5, 6, 8\}$. The block sequences obtained in each step are 
\[
\begin{split}
& \block{1} \\
& \block{1} \block{2} \\
& \block{3} \block{1} \block{2} \\
& \block{3} \block{4} \block{1} \block{2} \\
& \block{3\ \bar{5}} \block{4} \block{1} \block{2} \\
& \block{3\ \bar{5}\ \bar{6}} \block{4} \block{1} \block{2} \\
& \block{7} \block{3\ \bar{5}\ \bar{6}} \block{4} \block{1} \block{2} \\
& \block{7} \block{3\ \bar{5}\ \bar{6}} \block{4} \block{1\ \bar{8}} \block{2} \\
& \block{7} \block{3\ \bar{5}\ \bar{6}} \block{4} \block{1\ \bar{8}} \block{9} \block{2}
\end{split}
\]
Dissolving the block structure gives $\bpi = \phi(T) = 
7\ 3\ \bar{5}\ \bar{6}\ 4\ 1\ \bar{8}\ 9\ 2$. We note that
$P(\bpi) = \{5\}$, $Q(\bpi) = \{8\}$, and $\adjAsc(\bpi)=\{6,9\}$,
and hence the barred elements $X(\bpi) = \{5, 6, 8\}$ is a subset of
$P(\bpi) \cup Q(\bpi) \cup \adjAsc(\bpi) =
\{5, 6, 8, 9\}$. We also note that $4 \in \adjAsc(3\ 4\ 1\ 2)$,
but when $5$ is added, the short ascent pattern is destroyed.
Patterns on barred
elements can, on the other hand, never be destroyed.
\end{example}

There is an inverse $\phi^{-1}$: Given a barred
permutation $\bpi$ we construct a flat column-strict filling $T$ by letting
$\ell(T) = n(\bpi)$ and $X(T) = \barred(\bpi)$ and letting
$\alpha(T)_i$ be the number of entries weakly to the left of $i$ in
the single-row notation of $\pi$ that are not greater than $i$ and
do not belong to $\barred(\bpi)$.

\begin{theo}\label{th:Sbij}
The map $\phi\colon \flatcolstrictfillings \rightarrow \bcalS$
is a bijection.
Furthermore, we have
\begin{equation}\label{eq:Sdirect}
\phi(T\oplus T')=\phi(T)\oplus\phi(T')
\end{equation}
for any $T,T'\in\flatcolstrictfillings$, and if
$\phi(T)=\bpi$ the following equations hold.
\begin{subequations}
\begin{align}
n(T) &= n(\bpi), \label{eq:Sn} \\
\comp(T) &= \bcomp(\bpi), \label{eq:Scomp} \\
\Max(T) &= \Rmin(\bpi), \label{eq:Smax} \\
\Rmax(T) &= \Rmax(\bpi), \label{eq:Srmax} \\
X(T) &= \barred(\bpi), \label{eq:SX} \\
\Des(T)\cap X(T) &= P(\bpi)\cap \barred(\bpi), \label{eq:SP} \\
\Asc(T)\cap X(T) &= Q(\bpi)\cap \barred(\bpi), \label{eq:SQ} \\
\Rep(T)\cap X(T) &= \adjAsc(\bpi)\cap \barred(\bpi). \label{eq:SR} 
\end{align}
\end{subequations}
\end{theo}

\begin{proof}
First we must check that $\phi$ is well-defined.
Each barred entry of $\phi(T)$ is preceded by a smaller entry since it is
not the leftmost entry in its block and any greater entries in the block
are inserted to its right. Thus, the barred entries are ascents and
it follows that $\phi(T) \in \bcalS$.

That $\phi^{-1}$ as defined above really is an inverse of
$\phi$ follows from the observation that each block
contains exactly one entry that is not in $X(T)=\barred(\phi(T))$.

Turning to the set of equalities, \eqref{eq:Sn} and~\eqref{eq:SX} follow
directly from the definition of $\phi$, and~\eqref{eq:Scomp} is a consequence
of the relation~\eqref{eq:Sdirect} which is
straightforward to verify. To check~\eqref{eq:Smax} we note that a
maximal dot in column 
$i$ in $T$ will put the entry $i$ to the right of every entry less than
$i$ in $\bpi$. Hence $i \in \Rmin(\bpi)$. Conversely, if $i \in \Rmin(\bpi)$,
$\phi^{-1}$ will put the dot in the topmost cell of column $i$ in
$\phi^{-1}(\bpi)$, since $i$ has no smaller entries to its right in
$\bpi$. Equation~\eqref{eq:Srmax} is checked similarly:
If $i \in \Rmax(T)$, it will not have any larger
entries to its right in $\bpi$, hence $i \in \Rmax(\bpi)$, and vice versa.

For the final equalities, \eqref{eq:SP}, \eqref{eq:SQ}, and \eqref{eq:SR},
take any $i \in X(T)=\barred(\bpi)$.
If $i \in \Des(T)$, $\alpha_{i-1} > \alpha_i$ and the entry
$i$ is put to the left of the block containing $i-1$; hence $i \in 
P(\bpi)$. If $i \in \Asc(T)$, $\alpha_{i-1} < \alpha_i$ and $i$
is put in a block to the right of the block containing $i-1$; hence
$i \in Q(\bpi)$. Finally,
if $i \in\Rep(T)$, $\alpha_{i-1}=\alpha_i$ and $i$ is put
immediately after $i-1$ and in the same block; hence $i \in \adjAsc(\bpi)$.
\end{proof}

\begin{example}
We found in the previous example that $\bpi = \phi(T) =
7\ 3\ \bar{5}\ \bar{6}\ 4\ 1\ \bar{8}\ 9\ 2$.
Clearly,
$\Rmin(\bpi) = \{1, 2\} = \Max(T)$ and $\Rmax(\bpi) = \{2, 9\} = 
\Rmax(T)$. We also have 
\[
P(\bpi) \cap \barred(\bpi) = \{5\} \cap \{5, 6, 8\} = \{5\} = 
\{3, 5\} \cap \{5, 6, 8\} = \Des(T) \cap X(T),
\]
\[
Q(\bpi) \cap \barred(\bpi) = \{8\} \cap \{5, 6, 8\} = \{8\} = 
\{2, 4, 8, 9\} \cap \{5, 6, 8\} = \Asc(T) \cap X(T),
\]
\[
\adjAsc(\bpi) \cap \barred(\bpi) = \{6, 9\} \cap \{5, 6, 8\} = \{6\} = 
\{6, 7\} \cap \{5, 6, 8\} = \Rep(T) \cap X(T).
\]
\end{example}

\subsection{Hatted permutations}
\noindent
A \emph{silly ascent} of a permutation $\pi\in\calS_n$
is an entry $\pi(j)$ in the single-row notation
of $\pi$ such that $\pi(j-1)<\pi(j)$ or $j=1$, and $\pi(j)<n$.
So the set $\Ascsilly(\pi)$ of silly ascents is obtained from the set
of ordinary ascents
by replacing $n$ by $\pi(1)$ if $n$ is an ascent,
that is, $\Ascsilly(\pi)=(\Asc(\pi)\cup\{\pi(1)\})\setminus\{n\}$.
In particular, the
number of silly ascents always equals the number of ordinary ascents.

In analogy with the refinement $\Asc(\pi)=P(\pi)\cup Q(\pi)\cup \adjAsc(\pi)$
of ascents, we partition the set of silly ascents
into two types
$\Ascsilly(\pi)=\Psilly(\pi)\cup\Qsilly(\pi)$. Here, $\Psilly(\pi)$ is the
set of occurrences of the pattern \patternPsilly, that is, silly
ascents $\pi(j)$
such that $\pi(j)+1$ is to the left of $\pi(j)$ in the single-row notation
of $\pi$. $\Qsilly(\pi)$ is the set of silly ascents $\pi(j)$ such that
$\pi(j)+1$ is to the right of $\pi(j)$, or equivalently, the set of
occurrences of the pattern \patternQsilly in the permutation
$0\oplus\pi$ of $[0,n]$ obtained by inserting a zero at the very beginning
of the single-row notation of $\pi$.

\begin{defi}
A \emph{hatted permutation} $\hpi$ is a permutation
where zero or more silly ascents are marked by a hat.
The set of hatted ascents is denoted by
$\hatted(\hpi)\subseteq\Ascsilly(\hpi)$.

The set of hatted permutations is denoted by $\hcalS$.
\end{defi}

The \emph{direct sum} $\hpi\oplus\hpi'$ of two hatted permutations
is defined exactly as for barred permutations, and the number of
irreducible hatted components of $\hpi$ is denoted by
$\bcomp(\hpi)$.
Note that $\bcomp(\hpi)$ is not necessarily equal to $\comp(\hpi)$.
Example: $\bcomp(21\hat{3}574\hat{6})=
\bcomp(21\oplus \hat{1}352\hat{4})=2$ while
$\comp(2135746)=3$.

Let $\hattedplus(\hpi)=\hatted(\hpi)+1
=\{i+1\colon i\in\hatted(\hpi)\}$ and define
$\Psillyplus(\hpi)=\Psilly+1$ and
$\Qsillyplus(\hpi)=\Qsilly+1$ analogously.

We now define a map $\phisilly: \flatcolstrictfillings \rightarrow \hcalS$
from flat column-strict fillings to hatted permutations.
For a flat column-strict filling $T$, this
map iterates through the sequence $\alpha = \alpha(T)$ to create a
hatted permutation $\hpi$.

Start with the single block $\block{0}$. Now, for $i = 1, 2, \dotsc, 
n(T)$, if $i+1\in X(T)$ (i.e.~$\lambda(T)_i=\lambda(T)_{i+1}$)
insert $\hat{i}$ at the right end
of the $\alpha_i$th block from the left. If $i+1\notin X(T)$, insert the
new block $\block{i}$ immediately after the $\alpha_i$th block from
the left. Having reached $n(T)$, we dissolve the block structure and remove
the leading zero to obtain (the single-row notation of) $\hpi$.

\begin{example}
Using $\phisilly$, let us map the flat column-strict filling $T$ with shape
$\lambda(T) = (1, 2, 3, 4, 4, 4, 5, 5, 6)$ and
$\alpha = \alpha(T) = (1, 2, 1, 2, 1, 1, 1, 4, 5)$ into
$\hcalS$. Recall that the lazy set of $\lambda(T)$ is
$X(T) = \{5, 6, 8\}$. The block sequences obtained in each step are 
\[
\begin{split}
& \block{0} \\
& \block{0} \block{1} \\
& \block{0} \block{1} \block{2} \\
& \block{0} \block{3} \block{1} \block{2} \\
& \block{0} \block{3\ \hat{4}} \block{1} \block{2} \\
& \block{0\ \hat{5}} \block{3\ \hat{4}} \block{1} \block{2} \\
& \block{0\ \hat{5}} \block{6} \block{3\ \hat{4}} \block{1} \block{2} \\
& \block{0\ \hat{5}\ \hat{7}} \block{6} \block{3\ \hat{4}} \block{1} \block{2} \\
& \block{0\ \hat{5}\ \hat{7}} \block{6} \block{3\ \hat{4}} \block{1} \block{8} \block{2} \\
& \block{0\ \hat{5}\ \hat{7}} \block{6} \block{3\ \hat{4}} \block{1} \block{8} \block{9} \block{2}
\end{split}
\]
Removing the block sequence and the leading zero gives $\hpi = \phi(T) = 
\hat{5}\ \hat{7}\ 6\ 3\ \hat{4}\ 1\ 8\ 9\ 2$.
We note that
$\Psilly(\hpi) = \{4\}$ and $\Qsilly(\hpi) = \{5, 7, 8\}$, and hence
the hatted elements $\hatted(T) = \{4, 5, 7\}$
are a subset of $\Psilly(\hpi) \cup \Qsilly(\hpi) =
\{4, 5, 7, 8\}$. We also note that $3 \in \Qsilly(3\ \hat{4}\ 1\ 2)$,
but when $5$ is added, the pattern is destroyed. Patterns on hatted
elements can, on the other hand, never be destroyed.
\end{example}

There is an inverse $\phisilly^{-1}$: Given a hatted
permutation $\hpi$ we construct a flat column-strict filling $T$ by letting
$\ell(T) = n(\hpi)$ and $X(T) = \hattedplus(\hpi)$ and letting
$\alpha(T)_i$ be one plus the number of entries to the left of $i$ in
the single-row notation of $\pi$ that are less than $i$ and
do not belong to $\hatted(\hpi)$.

\begin{theo}\label{th:Ssillybij}
The map $\phisilly\colon \flatcolstrictfillings \rightarrow \hcalS$
is a bijection. Furthermore, we have
\begin{equation}\label{eq:Ssillydirect}
\phisilly(T\oplus T')=\phisilly(T)\oplus\phisilly(T')
\end{equation}
for any $T,T'\in\flatcolstrictfillings$, and
if $\phisilly(T)=\hpi$ the following equations hold.
\begin{subequations}
\begin{align}
n(T) &= n(\hpi), \label{eq:Ssillyn} \\
\comp(T) &= \bcomp(\hpi), \label{eq:Ssillycomp} \\
\Max(T) &= \Rmin(\hpi), \label{eq:Ssillymax} \\
\Rmax(T) &= \Rmax(\hpi), \label{eq:Ssillyrmax} \\
X(T) &= \hattedplus(\hpi), \label{eq:SsillyX} \\
\Des(T)\cap X(T) &= \Psillyplus(\hpi)\cap \hattedplus(\hpi),
\label{eq:SsillyP} \\
\bigl(\Asc(T)\cup\Rep(T)\bigr)\cap X(T) &= \Qsillyplus(\hpi)\cap \hattedplus(\hpi).
\label{eq:SsillyQ}
\end{align}
\end{subequations}
\end{theo}

\begin{proof}
First we must check that $\phisilly$ is well-defined.
Let $0\oplus\hpi$ denote the hatted permutation $\hpi$ preceded by a zero.
Each hatted entry of $0\oplus\hpi$ is preceded by a smaller entry since it is
not the leftmost entry in its block and any greater entries in the block
are inserted to its right. Thus, the hatted entries are silly ascents and
it follows that $\hpi \in \hcalS$.

That $\phisilly^{-1}$ as defined above really is an inverse of
$\phisilly$ follows from the observation that each block
contains exactly one entry that is not in $\hatted(\hpi)$.

Turning to the set of equalities, \eqref{eq:Ssillyn} and~\eqref{eq:SsillyX}
follow directly from the definition of $\phisilly$
and~\eqref{eq:Ssillycomp} is a consequence of~\eqref{eq:Ssillydirect} which
is straightforward to verify. To check~\eqref{eq:Ssillymax} we note that a
maximal dot in column 
$i$ in $T$ will put $i$ to the right of every entry less than
$i$. Hence $i \in \Rmin(\hpi)$. Conversely, if $i \in \Rmin(\hpi)$,
$\phisilly^{-1}$ will put the dot in the topmost cell of column $i$ in
$\phisilly^{-1}(\hpi)$, since $i$ has no smaller entries to its right in
$\hpi$. Equation~\eqref{eq:Ssillyrmax} is checked similarly:
If $i \in \Rmax(T)$, it will not have any larger
entries to its right in $\hpi$, hence $i \in \Rmax(\hpi)$, and vice versa. 

For the final equalities, \eqref{eq:SsillyP} and~\eqref{eq:SsillyQ},
since we know that $X(T) = \hattedplus(\hpi)$ it suffices
to show that for $i \in \hatted(T)$,
$i \in \Des(T)$ if and only if $i-1 \in
\Psilly(\hpi)$. If $i \in \Des(T)$,
$\alpha_{i-1} > \alpha_i$ and $i$ is
put to the left of the block containing $i-1$; hence $i-1 \in 
\Psilly(\hpi)$. If $i \notin \Des(T)$,
$\alpha_{i-1} \leq \alpha_i$ and $i$
is put in the same block as the block containing $i-1$ or to its right;
hence $i-1 \notin \Qsilly(\hpi)$.
\end{proof}

\begin{example}
We saw in the previous example that $\hpi = \phisilly(T) =
\hat{5}\ \hat{7}\ 6\ 3\ \hat{4}\ 1\ 8\ 9\ 2$. Clearly,
$\Rmin(\hpi) = \{1, 2\} = \Max(T)$ and $\Rmax(\hpi) = \{2, 9\} = 
\Rmax(T)$. We also have 
\[
\Psillyplus(\hpi) \cap \hattedplus(\hpi) = \{5\} \cap \{5, 6, 8\}
= \{5\} = 
\{3, 5\} \cap \{5, 6, 8\} = \Des(T) \cap X(T)
\]
and
\begin{multline*}
\Qsillyplus(\hpi) \cap \hattedplus(\hpi) = \{6, 8, 9\} \cap \{5, 6, 8\}
= \{6, 8\} \\
= \{2, 4, 6, 7, 8, 9\} \cap \{5, 6, 8\} =
\bigl(\Asc(T)\cup\Rep(T)\bigr)\cap X(T).
\end{multline*}
\end{example}

\section{Matchings and adjacencies}
\noindent
A {\em matching} of size $n$ is a set partition of $[2n]$ into parts
of size $2$. Each part is called an {\em arc}. For each arc
$A$, we refer to the smaller of its elements as the {\em opener}
$\opener(A)$ and the greater as the {\em closer} $\closer(A)$.
We denote the arcs of a matching $M$ by $M_1, M_2, \dotsc, M_n$ ordered
by closer so that $M_1$ is the arc with the smallest closer.

In the following, we will identify a set of drawn arcs between $2n$
points on a horizontal line with a matching $M$ by assigning the numbers
$1,2,\dotsc,2n$ to the points from left to right, see
Figure~\ref{fig:matching}.

\begin{figure}
\centering
\begin{tikzpicture}[line width=0.7pt, scale=0.4]
  \tikzstyle{disc} = [circle,thin,draw=black, minimum size=1.7pt,
    inner sep=0pt ]
  \draw (0.4, 0) -- (18.6, 0);
  \foreach \x in {1, ..., 18} \draw[fill=black] (\x, 0) circle (3pt);
  \foreach \x in {1, ..., 18} \draw (\x, -1) node{\scriptsize{\x}};
  \foreach \x/\a in {6/1, 9/2, 10/3, 12/4, 13/5, 14/6, 16/7, 17/8, 18/9}
    \draw (\x, -2) node {\scriptsize{\a}};
  \foreach \c/\r in {6/2.5,9/1,10/4,12/2,13/5,14/5,16/5.5,17/3,18/1.5}
    \draw (\c, 0) arc (0:180:\r cm);
\end{tikzpicture}
\caption{The matching 
\[
\left\{ \{1, 6\}, \{7, 9\}, \{2, 10\}, \{8, 12\}, \{3, 13\}, \{4,
14\}, \{5, 16\}, \{11, 17\}, \{15, 18\} \right\}.
\]
} \label{fig:matching} 
\end{figure}

A pair of arcs $\{A,B\}$ is called
\begin{itemize}
\item
a \emph{nesting} if
$\opener(B)<\opener(A)<\closer(A)<\closer(B)$,
\item
a \emph{crossing} if
$\opener(A)<\opener(B)<\closer(A)<\closer(B)$,
\item
a \emph{left adjacency} if $\abs{\opener(A)-\opener(B)}=1$,
\item
a \emph{right adjacency} if $\abs{\closer(A)-\closer(B)}=1$,
\item
a \emph{left (right) nesting (crossing)}
if it is both a left (right) adjacency and a nesting (crossing),
\item
a \emph{double crossing} if it is both a left and a right crossing, and
\item
a \emph{single crossing} if it is either a left or a right crossing.
\end{itemize}

Let us index the right adjacencies $\{M_{i-1},M_i\}$ by $i$,
that is the index of the arc with the rightmost closer, and the left
adjacencies by the index of the arc with
the leftmost opener. The set of (indices of) right and left
adjencencies are denoted by $\Radj(M)$ and $\Ladj(M)$, respectively.
Let $\Rne(M)$ and 
$\Rcr(M)$ denote the sets of (indices of) right nestings
and right crossings, respectively, and correspondingly, let $\Lcr(M)$
denote the sets of left crossings. Further, let
$\LRcr(M)$ denote the set of double crossings, indexed by their right arc.
In other words,
\begin{align*}
\Radj(M) &:=\{i\colon\closer(M_{i-1})=\closer(M_i)-1\}, \\
\Rne(M) &:=
\{i\colon \opener(M_i)<\opener(M_{i-1})<\closer(M_{i-1})=\closer(M_i)-1\} \\
\Rcr(M) &:=
\{i\colon
\opener(M_{i-1})<\opener(M_i)<\closer(M_{i-1})=\closer(M_i)-1\} \\
\Lcr(M) &:=
\{i\colon \exists j\colon\opener(M_i)=\opener(M_j)-1<\closer(M_i)<\closer(M_j)\} \\
\LRcr(M) &:=
\{i\colon \opener(M_{i-1})+1=\opener(M_i)<\closer(M_{i-1})=\closer(M_i)-1\}.
\end{align*}
We also let
$\Rcrproper(M)=\Rcr(M)\setminus\LRcr(M)$ and
$\Lcrproper(M)=\Lcr(M)\setminus(\LRcr(M)-1)$ denote the 
set of right and left single crossings, respectively,
and we let $\rne(M)$, $\rcr(M)$, $\lcr(M)$, $\lrcr(M)$, $\radj(M)$,
and $\ladj(M)$
denote the cardinalities of $\Rne(M)$, $\Rcr(M)$, $\Lcr(M)$, $\LRcr(M)$,
$\Radj(M)$, and $\Ladj(M)$.

Define $\Min(M) := \{i\colon
\opener(M_i) < \closer(M_1)\}$ and
note that the cardinality $\min(M)=\#\Min(M)$ coincides with
$\min(M)$ as defined in \cite{ClaLin2011}.

The set of matchings of size $n$ is denoted by $\calM_n$. We let $\calM =
\bigcup_{n \geq 1} \calM_n$ denote the set of all matchings. Of
special importance is the set $\calN_n$ of matchings with no left nesting
and the set $\calN =
\bigcup_{n \geq 1} \calN_n$ of all such matchings. The size of a
matching $M$ is denoted by $n(M)$.

Define the direct sum $M\oplus M'$ of two
matchings $M,M'\in\calM$ to be the matching obtained by
putting their arc diagrams side by side, $M$ to the left and $M'$
to the right, and then increasing the labels of $M'$ by $2n(M)$.
Figure~\ref{fig:matchingdirectsum} shows an example.
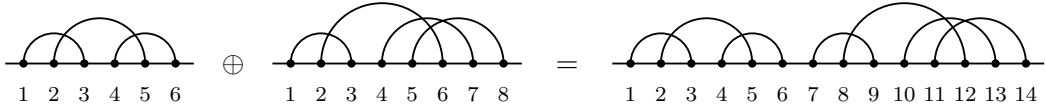
\begin{figure}
\centering
$
\raisebox{-1.32em}{
\begin{tikzpicture}[line width=0.7pt, scale=0.4]
  \tikzstyle{disc} = [circle,thin,draw=black, minimum size=1.7pt,
    inner sep=0pt ]
  \draw (0.4, 0) -- (6.6, 0);
  \foreach \x in {1, ..., 6} \draw[fill=black] (\x, 0) circle (3pt);
  \foreach \x in {1, ..., 6} \draw (\x, -1) node{\scriptsize{\x}};
  \foreach \c/\r in {3/1, 5/1.5, 6/1}
  \draw (\c, 0) arc (0:180:\r cm);
\end{tikzpicture}
}
\ \ \oplus\ \ 
\raisebox{-1.32em}{
  \begin{tikzpicture}[line width=0.7pt, scale=0.4]
   \tikzstyle{disc} = [circle,thin,draw=black, minimum size=1.7pt,
    inner sep=0pt ]
   \draw (0.4, 0) -- (8.6, 0);
   \foreach \x in {1, ..., 8} \draw[fill=black] (\x, 0) circle (3pt);
   \foreach \x in {1, ..., 8} \draw (\x, -1) node{\scriptsize{\x}};
   \foreach \c/\r in {3/1, 6/2, 7/1.5, 8/1.5}
     \draw (\c, 0) arc (0:180:\r cm);
 \end{tikzpicture}
}
\ \ =\ \ 
\raisebox{-1.32em}{
\begin{tikzpicture}[line width=0.7pt, scale=0.4]
  \tikzstyle{disc} = [circle,thin,draw=black, minimum size=1.7pt,
    inner sep=0pt ]
  \draw (0.4, 0) -- (14.6, 0);
  \foreach \x in {1, ..., 14} \draw[fill=black] (\x, 0) circle (3pt);
  \foreach \x in {1, ..., 14} \draw (\x, -1) node{\scriptsize{\x}};
  \foreach \c/\r in {3/1, 5/1.5, 6/1, 9/1, 12/2, 13/1.5, 14/1.5}
    \draw (\c, 0) arc (0:180:\r cm);
\end{tikzpicture}
}
$
\caption{Direct sum of two matchings. 
} \label{fig:matchingdirectsum} 
\end{figure}
Every matching $M$ can be written uniquely as a
direct sum of irreducible matchings, the number of which we
denote by $\comp(M)$.

\begin{example}
We will investigate the matching 
\[
M = \left\{ \{1, 6\}, \{7, 9\}, \{2, 10\}, \{8, 12\}, \{3, 13\}, \{4,
14\}, \{5, 16\}, \{11, 17\}, \{15, 18\} \right\}
\]
(see Figure \ref{fig:matching} for a pictorial presentation). It has
nine arcs, so $n(M) = 9$. The arcs are ordered by closers, indicated
in the bottom row of numbers. We have
$\Rne(M) = \{3, 5\}$, $\Rcr(M) = \{6, 8, 9\}$,
$\LRcr(M) = \{6\}$, $\Lcr(M) = \{1, 2, 3, 5, 6\}$,
and $\Min(M) = \{1, 3, 5, 6, 7\}$. Note
that $M$ is irreducible, so $\comp(M)=1$.
\end{example}

\begin{defi}
A \emph{marked matching} $\bM$ is a matching without left nestings
where zero or more of the right
adjacencies are marked. The set of (indices of) marked
adjencencies is denoted by $\markedadj(\bM)\subseteq\Radj(\bM)$.

The set of marked matchings is denoted by $\bcalN$. 
\end{defi}

There is a well-known bijection between strict fillings of
shapes of length $n$ and matchings on $[2n]$: Follow the border of the
shape from the south-west corner to the north-east corner and
label the edges $1,2,\dotsc,2n$ as in Figure~\ref{fig:matchingbij}.
Consider this to be a labelling of the rows and columns of the shape.
Now, construct a matching on $[2n]$ by drawing an arc between $i$ and $j$
if there is a dot in the row labelled $i$ and the column labelled $j$.

\begin{figure}
\centering
\begin{tikzpicture}[line width=0.7pt, scale=0.5]
  \tikzstyle{disc} = [circle,thin,draw=black, minimum size=1.7pt,
    inner sep=0pt ]
  \foreach \x/\y in {1/1,2/2,3/3,6/4,7/5} \draw[xshift=14pt,
    yshift=14pt ] (\x,\y) -- (9,\y);
  \foreach \x/\y in {1/1,2/2,3/3,4/4,5/4,6/4,7/5,8/5} \draw[xshift=14pt,
    yshift=14pt ] (\x,\y) -- (\x, 0); 
  \draw[xshift=14pt, yshift=14pt, thick] (0, 0) -- (9, 0) --
  (9, 6) -- (8, 6) -- (8, 5) -- (6, 5) -- (6, 4) -- (3, 4) --
  (3, 3) -- (2, 3) -- (2, 2)
  -- (1, 2) -- (1, 1) -- (0, 1) -- (0, 0);
  \foreach \x/\y in {1/1,2/2,3/1,4/2,5/1,6/1,7/1,8/4,9/5}
  \draw[fill=black] (\x,\y) circle (5pt);
\end{tikzpicture}
\hspace{1cm}
\begin{tikzpicture}[line width=0.7pt, scale=0.5]
  \tikzstyle{disc} = [circle,thin,draw=black, minimum size=1.7pt,
    inner sep=0pt ]
  \foreach \x/\y in {1/1,3/2,6/3} \draw[xshift=14pt,
    yshift=14pt ] (\x,\y) -- (9,\y);
  \foreach \x/\y in {1/1,2/2,3/2,4/3,5/3,6/3,7/4,8/4} \draw[xshift=14pt,
    yshift=14pt ] (\x,\y) -- (\x, 0); 
  \draw[xshift=14pt, yshift=14pt, thick] (0, 0) -- (9, 0) --
  (9, 4) -- (6, 4) -- (6, 3) --
  (3, 3) -- (3, 2)  -- (1, 2) -- (1, 1) -- (0, 1) -- (0, 0);
  \foreach \x/\y in {1/1,2/2,3/1,4/2,5/1,6/1,7/1,8/3,9/4}
  \draw[fill=black] (\x,\y) circle (5pt);
\end{tikzpicture}
\vskip0.5cm
\begin{tikzpicture}[line width=0.7pt, scale=0.5]
  \tikzstyle{disc} = [circle,thin,draw=black, minimum size=1.7pt,
    inner sep=0pt ]
  \foreach \x/\y in {0/1,0/2,0/3,0/4,1/5,1/6,3/7,6/8} \draw[xshift=14pt,
    yshift=14pt] (\x,\y) -- (9,\y);
  \foreach \x/\y in {1/5,2/7,3/7,4/8,5/8,6/8,7/9,8/9} \draw[xshift=14pt,
    yshift=14pt] (\x,\y) -- (\x, 0); 
  \draw[xshift=14pt, yshift=14pt, thick] (0, 0) -- (9, 0) --
  (9, 9) -- (6, 9) -- (6, 8) -- (3, 8) -- (3, 7) -- (1, 7) --
  (1, 5) -- (0, 5) -- (0, 0);
  \foreach \x/\y in {1/1,2/6,3/2,4/7,5/3,6/4,7/5,8/8,9/9}
  \draw[fill=black] (\x,\y) circle (5pt);
  \foreach \n/\x/\y in {1/0.2/1, 2/0.2/2, 3/0.2/3, 4/0.2/4, 5/0.2/5,
    6/0.8/5.8, 7/1.2/6.2, 8/1.2/7, 9/2/7.8, 10/2.8/7.8, 11/3.2/8.2,
    12/4/8.8, 13/5/8.8, 14/5.8/8.8, 15/6.2/9.2, 16/7/9.8, 17/8/9.8, 18/9/9.8}
  \draw[xshift=14pt,
    yshift=14pt] (\x cm-0.5cm,\y cm-0.5cm) node{{\tiny \n}}; 
\end{tikzpicture}
\caption{Upper left: The flat column-strict filling $T$ with
  $\lambda(T) = (1, 2, 3, 4,
  4, 4, 5, 5, 6)$ and $\alpha(T)=(1,2,1,2,1,1,1,4,5)$. Upper right:
$T'$ obtained by removing the empty rows from $T$. Bottom:
The unique strict filling $T''$ such that $\flatten(T'')=T'$.
The corresponding matching $\psi^{-1}(T')$ is the one shown in
Figure~\ref{fig:matching}.\label{fig:matchingbij}
}
\end{figure}
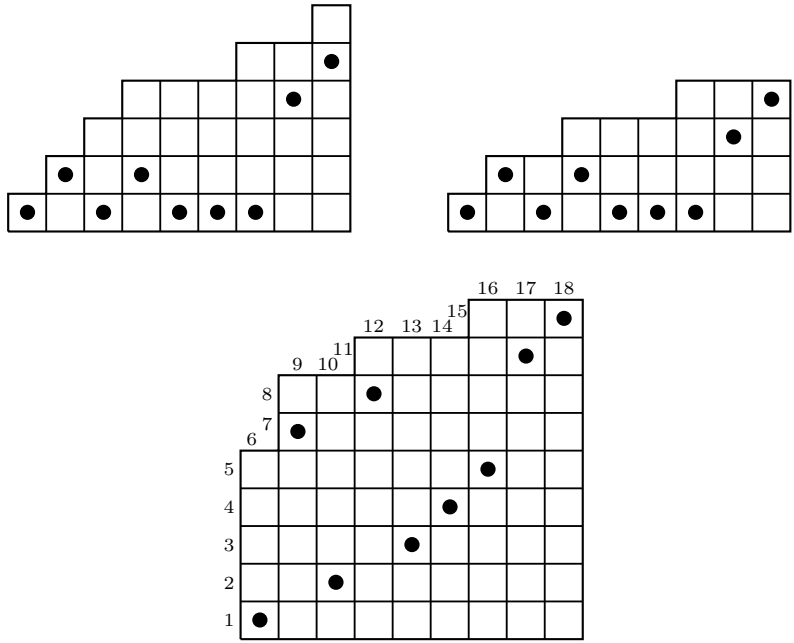

Under this bijection it is easy to see that left adjacencies
in the matching correspond to adjacent rows of
the same length in the
shape, and similarly right adjacencies
correspond to adjacent columns of the same length.
Furthermore, a pair of dots in adjacent rows or columns of the same length
correspond to a (left or right) crossing if they are ascending,
and to a (left or right) nesting if they are descending. (Recall that
a pair of dots are said to be descending (ascending) if one of them
is strictly to the south-west (north-west) of the other one.)
In particular, matchings without left nestings correspond to
strict fillings such that, for each $k$, the dots on rows of length $k$
form an ascending chain. Thus, via the flattening function
$T\mapsto\flatten(T)$
(i.e.~merging rows of equal length) we obtain a bijection $\psi$
from matchings without left nestings to
flat column-strict row-positive fillings,
see Figure~\ref{fig:matchingbij}.

Next we will define a function $f: \flatcolstrictfillings \rightarrow \bcalN$
from flat column-strict fillings to marked matchings.
Given a $T\in\flatcolstrictfillings$, remove
its empty rows to obtain a flat column-strict row-positive filling $T'$.
Let $M=f(T)$ be given by $M:=\psi^{-1}(T')$
together with the marks $\markedadj(M):=X(T)$.
Those marks are legal since
$X(T)\subseteq X(T')=\Radj(M)$.

\begin{theo}\label{th:Nbij}
The function $f\colon\flatcolstrictfillings\rightarrow\bcalN$ is a bijection.
Furthermore, we have
\begin{equation}\label{eq:Ndirect}
f(T\oplus T')=f(T)\oplus f(T')
\end{equation}
for any $T,T'\in\flatcolstrictfillings$, and
if $f(T) = \bM$ the following equations hold.
\begin{subequations}
\begin{align}
n(T) &= n(\bM), \label{eq:Nn} \\
\comp(T) &= \comp(\bM), \label{eq:Ncomp} \\
\Min(T) &= \Min(\bM), \label{eq:Nmin} \\
X(T) &= \markedadj(\bM), \label{eq:NX} \\
\Des(T) \cap X(T) &= \Rne(\bM) \cap \markedadj(\bM), \label{eq:NDes} \\
\Asc(T) \cap X(T) &= \Rcrproper(\bM) \cap\markedadj(\bM), \label{eq:NAsc} \\
\Rep(T) \cap X(T) &= \LRcr(\bM) \cap \markedadj(\bM). \label{eq:NRep}
\end{align}
\end{subequations}
\end{theo}

\begin{proof}
We know from above that $\psi$ is a bijection and it is easy to see that
for any subset $X$ of $X(T')$ there is a unique way of inserting empty
rows into $T'$ to obtain a flat column-strict filling $T$ with $X(T)=X$.
This shows that $f$ is a bijection.

Turning to the set of equalities, \eqref{eq:Nn} and~\eqref{eq:NX}
follow directly from the definition of $f$, and~\eqref{eq:Ncomp} is a
consequence of~\eqref{eq:Ndirect} which is straightforward to verify.
The remaining equations, \eqref{eq:Nmin},~\eqref{eq:NDes},~\eqref{eq:NAsc},
and~\eqref{eq:NRep} are also just matters of inspection.
\end{proof}

\begin{example}
Let $T$ be the flat column-strict filling from our previous examples.
We have previously seen that $X(T) = \{5, 6, 8\}$ and for $\bM =
f(T)$, $\markedadj(\bM)$ is the same set. We have $\Min(\bM) =
\{1, 3, 5, 6, 7\} = \Min(T)$. With $\Rne(\bM) = \{3, 5\}$,
$\Rcr(\bM) = \{6, 8, 9\}$, and $\LRcr(\bM)=\{6\}$ we find that 
\[
\Rne(\bM) \cap \markedadj(\bM) = \{3, 5\} \cap \{5, 6, 8\} = \{5\} 
= \{3, 5\} \cap \{5, 6, 8\} = \Des(T) \cap X(T),
\]
\[
\Rcrproper(\bM) \cap \markedadj(\bM)
= \{8, 9\} \cap \{5, 6, 8\} = \{8\} =  
\{2, 4, 8, 9\} \cap \{5, 6, 8\} = \Asc(T) \cap X(T),
\]
\[
\LRcr(\bM) \cap \markedadj(\bM) = \{6\} \cap \{5, 6, 8\} = \{6\} 
= \{6, 7\} \cap \{5, 6, 8\} = \Rep(T) \cap X(T).
\]
\end{example}

Equations~\eqref{eq:SP}, \eqref{eq:SQ}, \eqref{eq:SR},
\eqref{eq:SsillyP}, \eqref{eq:SsillyQ},
\eqref{eq:NDes}, \eqref{eq:NAsc}, and~\eqref{eq:NRep} all include
an intersection with $X(T)$.
Looking back at the examples, we note that without taking this intersection
the equations would not hold.
In the following sections we will use inclusion-exclusion to prove that
while the bijections above seem to indicate the opposite,
there are other bijections that allow the restriction to $X(T)$ to
be dropped.

\section{Restriction to right nestings}
\label{sc: conjecture 20}
\noindent

Later on, in Section~\ref{sc: generating functions}, we will
take full advantage of the bijections $\phi$, $\phisilly$, and $f$
from the previous sections. As a warm-up, in this section we will
confine ourselves to hatted permutations and matchings and we will
forget about crossings and consider only right nestings. In the end
we will obtain a proof of Conjecture~20 in~\cite{ClaLin2011}.

Define the following three sets.
\begin{align*}
\flatcolstrictfillingsrestr
&:= \{T\in\flatcolstrictfillings\colon\Des(T)\supseteq X(T)\}
\subset\flatcolstrictfillings, \\
\hcalSrestr &:=
\{\hpi\in\hcalS\colon \Psilly(\hpi)\supseteq \hatted(\hpi)\}
\subset\hcalS, \\
\bcalNrestr &:= \{\bM\in\bcalN\colon\Rne(\bM)\supseteq\markedadj(\bM)\}
\subset\bcalN.
\end{align*}
In other words, $\flatcolstrictfillingsrestr$ is the set of flat column-strict
fillings such that, for any $k$, the dots in columns of length $k$ form
a descending chain;
$\hcalSrestr$ is the set of
hatted permutations where only occurrences
of \patternPsilly are hatted;
and $\bcalNrestr$ is the set of marked matchings
where only right nestings are marked.

\begin{lemma}\label{lm:restrbij}
The maps $\phisilly$ and $f$ from
Theorems~\ref{th:Ssillybij} and~\ref{th:Nbij}
restricted to
$\flatcolstrictfillingsrestr$ are bijections
$\flatcolstrictfillingsrestr\rightarrow\hcalSrestr$
and $\flatcolstrictfillingsrestr\rightarrow\bcalNrestr$, respectively.
Moreover, the equations
$\comp(\phisilly(T))=\comp(f(T))=\comp(T)$ hold for any
$T\in\flatcolstrictfillingsrestr$.
\end{lemma}
\begin{proof}
The first assertion follows immediately from Theorems~\ref{th:Ssillybij}
and~\ref{th:Nbij}, and the second assertion is clear after noting that
$\bcomp(\hpi)=\comp(\hpi)$ for any $\hpi\in\hcalSrestr$.
\end{proof}

Define
\begin{align*}
\bSsillyrestr(x,t,s,r) &:=
\sum_{n=1}^\infty t^n \sum_{\pi\in\calS_n}
r^{\comp(\pi)}s^{\rmax(\pi)}x^{\psilly(\pi)}, \\
\bNrestr(x,t,s,r) &:=
\sum_{n=1}^\infty t^n\sum_{M\in\calN_n}
r^{\comp(M)}s^{\min(M)}x^{\rne(M)}.
\end{align*}
Our goal in this section is to show that these generating functions
are equal and to obtain an expression for them.

Let us first make a variable substitution, and define
\begin{align}
\Ssillyrestr(u,t,s,r) &:= \bSsillyrestr(u+1,t,s,r),
\label{eq:Ssillyrestrsubst} \\
\Nrestr(u,t,s,r) &:= \bNrestr(u+1,t,s,r). \label{eq:Nrestrsubst}
\end{align}

\begin{prop}
The following identities hold.
\begin{align*}
\Ssillyrestr(u,t,s,r) &= \sum_{\hpi\in\hcalSrestr}
t^{n(\hpi)}s^{\rmax(\hpi)}r^{\comp(\hpi)}u^{\#\hatted(\hpi)}, \\
\Nrestr(u,t,s,r) &= \sum_{\bM\in\bcalNrestr}
t^{n(\bM)}s^{\min(\bM)}r^{\comp(\bM)}u^{\#\markedadj(\bM)}
\end{align*}
\end{prop}
\begin{proof}
Since an element $\hpi\in\hcalSrestr$ is essentially a permutation
together with a subset $\hatted(\hpi)$ of $\Psilly(\hpi)$, the sieve principle
yields the desired result, see e.g.~\cite{StanleyII}.
The same holds for matchings.
\end{proof}

The steepening function $T\mapsto\steepen(T)$ is clearly a bijection
from $\flatcolstrictfillingsrestr$ to $\staircasecolposfillings$,
the set of column-positive staircase fillings.
Moreover,
it preserves the statistics $n$, $\min$, $\max$, $\rmax$, and $\comp$,
and $\#X(T)=n(\steepen(T))-\ell(\steepen(T))$ for any
$T\in\flatcolstrictfillingsrestr$.
These facts together with Lemma~\ref{lm:restrbij} yield the following theorem.
\begin{theo}\label{th:SINrestr}
The identities
\begin{align*}
\Ssillyrestr(u,t,s,r) &= \Irestr^{\rmax}(u,t,s,r), \\
\Nrestr(u,t,s,r) &= \Irestr^{\min}(u,t,s,r)
\end{align*}
hold, where we define the generating functions
\[
\Irestr^\mu(u,t,s,r)
:= \sum_{T \in \staircasecolposfillings}
t^{n(T)}s^{\mu(T)}r^{\comp(T)}
u^{n(T)-\ell(T)}.
\]
for $\mu\in\{\min,\rmax\}$.
\end{theo}

It remains to show that $\Irestr^{\rmax}(u,t,s,r)=\Irestr^{\min}(u,t,s,r)$
and to compute this generating function.
To show equality we will present an involution $\iota$ on
$\staircasecolposfillings$
taking $\rmax$ to $\min$ while preserving the statistics $n$, $\ell$,
and $\comp$.

An \emph{enriched permutation filling} is a positive filling of a
square shape with the property that each dot is
the leftmost dot in its row if and only if it is the topmost dot in its
column. Let $\enpermfillings$ denote the set of enriched permutations.
(Note that the dot diagram of a permutation $\pi$ (where there is a dot at
$(i,j)$ if $\pi(i)=j$) is indeed an enriched permutation filling.)

Given two enriched permutation fillings $\rho,\rho'\in\enpermfillings$,
we define their \emph{boxed sum} $\rho\boxplus\rho'$ as the
enriched permutation filling obtained by putting $\rho'$ to the
north-east of of $\rho$ and then filling out with empty rectangles
at the north-west and south-east corners in order to get a square shape.
\begin{figure}
\centering
$
\raisebox{-1em}{
\begin{tikzpicture}[line width=0.7pt, scale=0.5]
  \tikzstyle{disc} = [circle,thin,draw=black, minimum size=1.7pt,
    inner sep=0pt ]
  \foreach \x/\y in {0/1,0/2} \draw[xshift=14pt,
    yshift=14pt ] (\x,\y) -- (3,\y);
  \foreach \x/\y in {1/3,2/3} \draw[xshift=14pt,
    yshift=14pt ] (\x,\y) -- (\x, 0); 
  \draw[xshift=14pt, yshift=14pt, thick] (0, 0) -- (3, 0) --
  (3, 3) -- (0, 3) -- (0, 0);
  \foreach \x/\y in {1/2,2/2,2/3,3/1}
  \draw[fill=black] (\x,\y) circle (5pt);
\end{tikzpicture}
}
\ \ \oplus\ \ 
\raisebox{-1em}{
\begin{tikzpicture}[line width=0.7pt, scale=0.5]
  \tikzstyle{disc} = [circle,thin,draw=black, minimum size=1.7pt,
    inner sep=0pt ]
  \foreach \x/\y in {0/1} \draw[xshift=14pt,
    yshift=14pt ] (\x,\y) -- (2,\y);
  \foreach \x/\y in {1/2} \draw[xshift=14pt,
    yshift=14pt ] (\x,\y) -- (\x, 0); 
  \draw[xshift=14pt, yshift=14pt, thick] (0, 0) -- (2, 0) --
  (2, 2) -- (0, 2) -- (0, 0);
  \foreach \x/\y in {1/1,2/1,2/2}
  \draw[fill=black] (\x,\y) circle (5pt);
\end{tikzpicture}
}
\ \ =\ \ 
\raisebox{-1em}{
\begin{tikzpicture}[line width=0.7pt, scale=0.5]
  \tikzstyle{disc} = [circle,thin,draw=black, minimum size=1.7pt,
    inner sep=0pt ]
  \foreach \x/\y in {0/1,0/2,0/3,0/4} \draw[xshift=14pt,
    yshift=14pt ] (\x,\y) -- (5,\y);
  \foreach \x/\y in {1/5,2/5,3/5,4/5} \draw[xshift=14pt,
    yshift=14pt ] (\x,\y) -- (\x, 0); 
  \draw[xshift=14pt, yshift=14pt, thick] (0, 0) -- (5, 0) --
  (5, 5) -- (0, 5) -- (0, 0);
  \foreach \x/\y in {1/2,2/2,2/3,3/1,4/4,5/4,5/5}
  \draw[fill=black] (\x,\y) circle (5pt);
\end{tikzpicture}
}
$
\caption{Boxed sum of enriched permutations.} \label{fig:boxedsum} 
\end{figure}
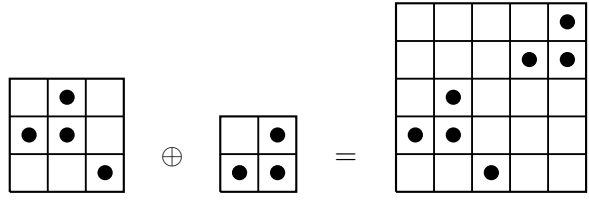
Enriched permutation fillings that cannot be written
as a boxed sum is called \emph{box-irreducible}.
Every enriched permutation $\rho$ can be
written uniquely as
a boxed sum of box-irreducible components, and the number of such
components is denoted by $\boxcomp(\rho)$.

We will define a map $g$ from column-positive staircase fillings
to enriched permutation fillings recursively as follows.
\begin{itemize}
\item
The filling
\raisebox{-0.1em}{
\begin{tikzpicture}[line width=0.7pt, scale=0.3]
  \tikzstyle{disc} = [circle,thin,draw=black, minimum size=1.7pt,
    inner sep=0pt ]
  \draw[xshift=14pt, yshift=14pt, thick] (0, 0) -- (1, 0) --
  (1, 1) -- (0, 1) -- (0, 0);
  \foreach \x/\y in {1/1}
  \draw[fill=black] (\x,\y) circle (5pt);
\end{tikzpicture}
}
is mapped to itself.
\item
If $T\in\staircasecolposfillings$ has more than one column,
remove the rightmost
column from $T$ to obtain $T'\in\staircasecolposfillings$,
and let $\rho'=g(T')$.
Now we obtain $\rho=g(T)$ by inserting an empty row of length $\ell(\rho')$
into $\rho'$ at position $\max\alpha(T)_{\ell(T)}$ and then inserting the
rightmost column of $T$ to the far right.
\end{itemize}

It is easy to see that $g$ is a bijection and the inverse map $g^{-1}$
from enriched permutation fillings to column-positive staircase fillings
is straightforward:
Given an enriched permutation filling $\rho$, delete every cell
in $\rho$ that is to the left of the leftmost dot in its row.
Then, flush the remaining
cells downwards to obtain the column-positive staircase filling $g^{-1}(\rho)$.

The bijection $g$ transfers many useful properties as stated by
the following theorem whose proof is just a matter of straightforward
verification and therefore omitted.
\begin{theo}
The bijection $g\colon\staircasecolposfillings\rightarrow\enpermfillings$
has the property
\begin{equation}\label{eq:boxedsum}
g(T\oplus T')=g(T)\boxplus g(T')
\end{equation}
for any $T,T'\in\staircasecolposfillings$, and
if $\rho=g(T)$ the following equations hold.
\begin{align*}
\ell(T) &= \ell(\rho) \\
n(T) &= n(\rho) \\
\comp(T) &= \boxcomp(\rho) \\
\min(T) &= \lmin(\rho) \\
\rmax(T) &= \rmax(\rho)
\end{align*}
\end{theo}

For any enriched permutation filling $\rho$, define the
\emph{transpose} of $\rho$, denoted by $\rho^T$, to be the enriched permutation
filling obtained by reflecting $\rho$ in the north-west to south-east diagonal.
Clearly, transposition is an involution on $\enpermfillings$ with the
properties that $\lmin(\rho)=\rmax(\rho^T)$ and
$\boxcomp(\rho)=\boxcomp(\rho^T)$.
 
We define the involution $\iota$ on $\staircasecolposfillings$
by $\iota(T)=g^{-1}(g(T)^T)$. Clearly, $\iota$
has the promised properties and we conclude that
\begin{equation}\label{eq:IrmaxequalsImin}
\Irestr^{\rmax}(u,t,s,r)=\Irestr^{\min}(u,t,s,r).
\end{equation}
Figure~\ref{fig:iota} gives an example of the situation.
\begin{figure}
\centering
\begin{tabular}{ccccc}
\tiny
$\begin{array}{ccc}
\min & = & 3 \\
\rmax & = & 4
\end{array}$
& &
\tiny
$\begin{array}{ccc}
\min & = & 3 \\
\rmax & = & 4
\end{array}$
& &
\tiny
$\begin{array}{ccc}
\lmin & = & 3 \\
\rmax & = & 4
\end{array}$ \\
\begin{tikzpicture}[line width=0.7pt, scale=0.5]
  \tikzstyle{disc} = [circle,thin,draw=black, minimum size=1.7pt,
    inner sep=0pt ]
  \foreach \x/\y in {1/1,2/2,4/3,5/4} \draw[xshift=14pt,
    yshift=14pt ] (\x,\y) -- (7,\y);
  \foreach \x/\y in {1/1,2/2,3/3,4/3,5/4,6/5} \draw[xshift=14pt,
    yshift=14pt ] (\x,\y) -- (\x, 0); 
  \draw[xshift=14pt, yshift=14pt, thick] (0, 0) -- (7, 0) --
  (7, 5) -- (5, 5) -- (5, 4) -- (4, 4) -- (4, 3) -- (2, 3) --
  (2, 2) -- (1, 2) -- (1, 1) -- (0, 1) -- (0, 0);
  \foreach \x/\y in {1/1,2/1,3/3,4/1,5/3,6/3,7/2}
  \draw[fill=black] (\x,\y) circle (5pt);
\end{tikzpicture}
&
\raisebox{1cm}{$\overset{\steepen{}}{\longmapsto}$}
&
\begin{tikzpicture}[line width=0.7pt, scale=0.5]
  \tikzstyle{disc} = [circle,thin,draw=black, minimum size=1.7pt,
    inner sep=0pt ]
  \foreach \x/\y in {1/1,2/2,3/3,4/4} \draw[xshift=14pt,
    yshift=14pt ] (\x,\y) -- (5,\y);
  \foreach \x/\y in {1/1,2/2,3/3,4/4} \draw[xshift=14pt,
    yshift=14pt ] (\x,\y) -- (\x, 0); 
  \draw[xshift=14pt, yshift=14pt, thick] (0, 0) -- (5, 0) --
  (5, 5) -- (4, 5) -- (4, 4) -- (3, 4) -- (3, 3) -- (2, 3) --
  (2, 2) -- (1, 2) -- (1, 1) -- (0, 1) -- (0, 0);
  \foreach \x/\y in {1/1,2/1,3/1,3/3,4/3,5/2,5/3}
  \draw[fill=black] (\x,\y) circle (5pt);
\end{tikzpicture}
&
\raisebox{1cm}{$\overset{g}{\longmapsto}$}
&
\begin{tikzpicture}[line width=0.7pt, scale=0.5]
  \tikzstyle{disc} = [circle,thin,draw=black, minimum size=1.7pt,
    inner sep=0pt ]
  \foreach \x/\y in {0/1,0/2,0/3,0/4} \draw[xshift=14pt,
    yshift=14pt ] (\x,\y) -- (5,\y);
  \foreach \x/\y in {1/5,2/5,3/5,4/5} \draw[xshift=14pt,
    yshift=14pt ] (\x,\y) -- (\x, 0); 
  \draw[xshift=14pt, yshift=14pt, thick] (0, 0) -- (5, 0) --
  (5, 5) -- (0, 5) -- (0, 0);
  \foreach \x/\y in {1/2,2/1,3/1,3/5,4/4,5/2,5/3}
  \draw[fill=black] (\x,\y) circle (5pt);
\end{tikzpicture}
\\
& & & & \raisebox{0cm}[0.5cm][0.5cm]{$\Big\updownarrow$\footnotesize transposition} \\
\begin{tikzpicture}[line width=0.7pt, scale=0.5]
  \tikzstyle{disc} = [circle,thin,draw=black, minimum size=1.7pt,
    inner sep=0pt ]
  \foreach \x/\y in {1/1,2/2,3/3,5/4} \draw[xshift=14pt,
    yshift=14pt ] (\x,\y) -- (7,\y);
  \foreach \x/\y in {1/1,2/2,3/3,4/4,5/4,6/5} \draw[xshift=14pt,
    yshift=14pt ] (\x,\y) -- (\x, 0); 
  \draw[xshift=14pt, yshift=14pt, thick] (0, 0) -- (7, 0) --
  (7, 5) -- (5, 5) -- (5, 4) -- (3, 4) -- (3, 3) -- (2, 3) --
  (2, 2) -- (1, 2) -- (1, 1) -- (0, 1) -- (0, 0);
  \foreach \x/\y in {1/1,2/1,3/1,4/4,5/1,6/4,7/3}
  \draw[fill=black] (\x,\y) circle (5pt);
\end{tikzpicture}
&
\raisebox{1cm}{$\overset{\steepen{}}{\longmapsto}$}
&
\begin{tikzpicture}[line width=0.7pt, scale=0.5]
  \tikzstyle{disc} = [circle,thin,draw=black, minimum size=1.7pt,
    inner sep=0pt ]
  \foreach \x/\y in {1/1,2/2,3/3,4/4} \draw[xshift=14pt,
    yshift=14pt ] (\x,\y) -- (5,\y);
  \foreach \x/\y in {1/1,2/2,3/3,4/4} \draw[xshift=14pt,
    yshift=14pt ] (\x,\y) -- (\x, 0); 
  \draw[xshift=14pt, yshift=14pt, thick] (0, 0) -- (5, 0) --
  (5, 5) -- (4, 5) -- (4, 4) -- (3, 4) -- (3, 3) -- (2, 3) --
  (2, 2) -- (1, 2) -- (1, 1) -- (0, 1) -- (0, 0);
  \foreach \x/\y in {1/1,2/1,3/1,4/1,4/4,5/3,5/4}
  \draw[fill=black] (\x,\y) circle (5pt);
\end{tikzpicture}
&
\raisebox{1cm}{$\overset{g}{\longmapsto}$}
&
\begin{tikzpicture}[line width=0.7pt, scale=0.5]
  \tikzstyle{disc} = [circle,thin,draw=black, minimum size=1.7pt,
    inner sep=0pt ]
  \foreach \x/\y in {0/1,0/2,0/3,0/4} \draw[xshift=14pt,
    yshift=14pt ] (\x,\y) -- (5,\y);
  \foreach \x/\y in {1/5,2/5,3/5,4/5} \draw[xshift=14pt,
    yshift=14pt ] (\x,\y) -- (\x, 0); 
  \draw[xshift=14pt, yshift=14pt, thick] (0, 0) -- (5, 0) --
  (5, 5) -- (0, 5) -- (0, 0);
  \foreach \x/\y in {1/3,2/2,3/1,4/1,4/5,5/3,5/4}
  \draw[fill=black] (\x,\y) circle (5pt);
\end{tikzpicture} \\
\tiny
$\begin{array}{ccc}
\min & = & 4 \\
\rmax & = & 3
\end{array}$
& &
\tiny
$\begin{array}{ccc}
\min & = & 4 \\
\rmax & = & 3
\end{array}$
& &
\tiny
$\begin{array}{ccc}
\lmin & = & 4 \\
\rmax & = & 3
\end{array}$
\end{tabular}
\caption{The fillings to the left belong to $\flatcolstrictfillingsrestr$
and are steepened to the column-positive staircase fillings in the middle.
These in turn are mapped by $g$ to the enriched permutations to the
right which are each others' transposes.} \label{fig:iota}
\end{figure}
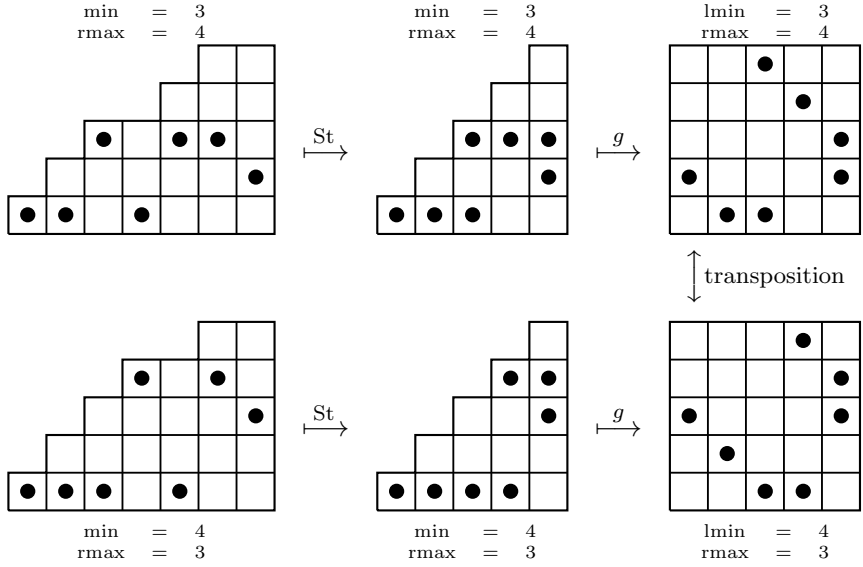

Now, we turn to the task of computing $\Irestr^{\min}(u,t,s,r)$.
Let $\Irestr^{\min}_{\irr}(u,t,s)
=\lim_{r\rightarrow0}r^{-1}\Irestr^{\min}(u,t,s,r)$
be the corresponding generating function that keeps track only of
irreducible fillings.
Since only the first irreducible component of a filling $T$ contributes
to $\min(T)$, we have
\begin{align}
\Irestr^{\min}(u,t,s,r) &=
r\Irestr^{\min}_{\irr}(u,t,s)\bigl(1+\Irestr^{\min}(u,t,1,r)\bigr),
\label{eq:Irestr1} \\
\Irestr^{\min}(u,t,s,r)
&=\frac{r\Irestr^{\min}_{\irr}(u,t,s)}{1-r\Irestr^{\min}_{\irr}(u,t,1)}.
\label{eq:Irestr2}
\end{align}
Plugging $r=1$ into~\eqref{eq:Irestr1} yields
\begin{equation}\label{eq:Iirr}
\Irestr^{\min}_{\irr}(u,t,s)=
\frac{\Irestr^{\min}(u,t,s,1)}{1+\Irestr^{\min}(u,t,1,1)}
=\frac{J(s)}{1+J(1)},
\end{equation}
where we have put $J(s):=\Irestr^{\min}(u,t,s,1)$
for notational convenience.
Now plugging this into~\eqref{eq:Irestr2}, we obtain
\begin{equation}\label{eq:Irestr3}
\Irestr^{\min}(u,t,s,r)
=\frac{r\frac{J(s)}{1+J(1)}}{1-r\frac{J(1)}{1+J(1)}}
=\frac{rJ(s)}{1+(1-r)J(1)}.
\end{equation}
To compute $J(s)=\Irestr^{\min}(u,t,s,1)$ we construct a typical
column-positive staircase filling by first choosing the number $\ell$
of columns and then, for $k=1,2,\dotsc,\ell$, choosing which
cells in the $k$th column should have a dot. Translated to
a generating function, this procedure becomes
\[
J(s)=\Irestr^{\min}(u,t,s,1)
=\sum_{\ell=1}^\infty\prod_{k=1}^\ell
u^{-1}\bigl((1+tu)^{k-1}(1+stu)-1\bigr).
\]
and plugging this into~\eqref{eq:Irestr3} finally yields
\[
\Irestr^{\min}(u,t,s,r)
=\frac{rJ(s)}{1+(1-r)J(1)}
=\frac{r\sum_{m=1}^\infty\prod_{k=1}^m
u^{-1}\bigl((1+tu)^{k-1}(1+stu)-1\bigr)}%
{1+(1-r)\sum_{m=1}^\infty\prod_{k=1}^m
u^{-1}\bigl((1+tu)^k-1\bigr)}.
\]

From~\eqref{eq:IrmaxequalsImin} and
Theorem~\ref{th:SINrestr} we see that substituting $x-1$ for $u$ in the
last equation,
according to \eqref{eq:Ssillyrestrsubst} and
\eqref{eq:Nrestrsubst}, yields the following theorem.

\begin{theo}\label{th:conjtwenty}
\begin{multline*}
\sum_{n=1}^\infty t^n \sum_{\pi\in\calS_n}
r^{\comp(\pi)}s^{\rmax(\pi)}x^{\psilly(\pi)}
= \sum_{n=1}^\infty t^n\sum_{M\in\calN_n}
r^{\comp(M)}s^{\min(M)}x^{\rne(M)} \\
= \frac{r\sum_{m=1}^\infty\prod_{k=1}^m
(x-1)^{-1}\bigl((1+t(x-1))^{k-1}(1+st(x-1))-1\bigr)}%
{1+(1-r)\sum_{m=1}^\infty\prod_{k=1}^m
(x-1)^{-1}\bigl((1+t(x-1))^k-1\bigr)}
\end{multline*}
\end{theo}
The first equation of the theorem proves Claesson and Linusson's
Conjecture~20 in \cite{ClaLin2011} after two simple observations:
\begin{itemize}
\item
The pattern $p$ used by Claesson and Linusson is our pattern
$\psilly=$\patternPsillyplain rotated 180 degrees. Hence our $\rmax(\pi)$
corresponds to their $\lmin(\pi)$.
\item
The part of their conjecture stating that $(\rne, \comp, \min)$
is equidistributed on matchings and posets follows from their
Proposition~17 together with the fact that
their bijection $h$ preserves $\rne$.
\end{itemize}

\section{Noncommutative formal power series} \label{sc: NFPS}
\noindent
For the rest of this paper we will continue concerning ourselves with
permutation patterns and matching adjacencies. However,
we will not be interested merely in their
\emph{numbers} but in their \emph{positions} inside the permutations and
matchings as well, so
ordinary generating functions will not suffice anymore.
Instead, our results will
typically be expressed in terms of noncommutative power series,
sometimes mixed up with commutative variables.

For noncommutative variables
we will use bold symbols, like $\bx$ and $\by$, while the commutative
ones will be thin, like $x$ and $y$.
The ring $\formal{\C[[x_1,\dotsc,x_m]]}{\bx_1,\dotsc,\bx_n}$
of formal series over the complex numbers in the variables
$x_1,\dotsc,x_m$ (commutative) and $\bx_1,\dotsc,\bx_n$ (noncommutative)
is an object that requires some care,
so in this section we will define it precisely and discuss some of its
properties. Assuming that the reader is familiar with the
integral domain $\C[[x_1,\dotsc,x_m]]$ of ordinary commutative formal
series in variables $x_1,\dotsc,x_m$, we will define
$\formal{R}{\bx_1,\dotsc,\bx_n}$ for any integral domain $R$.
We will follow Stanley's exposition in~\cite[Sec.~6.5]{StanleyII} closely,
but Stanley deals only with the case where $R$ is a field.

Let $\bX=\{\bx_1,\dotsc,\bx_n\}$ be a set of noncommutative variables
and let $\bX^\ast$ be the free monoid generated by $\bX$. Thus
$\bX^\ast$ consists of all finite strings (including the empty string 1)
of letters in $\bX$.

\begin{defi}
A \emph{formal (power) series} in $\bX$ over an integral domain
$R$ is a function
$S\colon\bX^\ast\rightarrow R$. We write $\inner{S}{w}$ for $S(w)$ and then
write
\[
S=\sum_{w\in\bX^\ast}\inner{S}{w}w.
\]
The set of all formal series in $\bX$ is denoted $\formal{R}{\bX}$.
\end{defi}

We identify $1\in\bX^\ast$ with $1\in R$ and abbreviate the
term $\alpha\cdot 1$ of the above series $S$ as $\alpha$.
To make $\formal{R}{\bX}$ a ring, we define addition and multiplication
in the obvious way:
\begin{align*}
S+T &= \sum_{w\in\bX^\ast}(\inner{S}{w}+\inner{T}{w})w, \\
ST &= \sum_{u,v\in\bX^\ast}\inner{S}{u}\inner{T}{v}uv
=\sum_{w\in\bX^\ast}\left(\sum_{uv=w}\inner{S}{u}\inner{T}{v}\right)w.
\end{align*}

A sequence $S_1,S_2,\dotsc$ of formal series is said to \emph{converge} to a
formal series $S$ if for all $w\in\bX^\ast$ the sequence
$\inner{S_1}{w},\inner{S_2}{w},\dotsc$ has only finitely many
terms unequal to $\inner{S}{w}$.

\begin{theo}
$S$ is invertible in $\formal{R}{\bX}$ if and only if
$\inner{S}{1}$ is invertible in $R$.
\end{theo}
\begin{proof}
The ``only if'' direction follows from the fact that
$\inner{ST}{1}=\inner{TS}{1}=\inner{S}{1}\inner{T}{1}$.
For the ``if'' direction,
suppose $\inner{S}{1}=\alpha$ is invertible in $R$ and define
\[
T=\alpha^{-1}\sum_{k=0}^\infty(1-\alpha^{-1}S)^k.
\]
This sum converges formally and it is easy to check that $ST=TS=1$.
\end{proof}

It is evident from the definition of addition and multiplication
on $\formal{R}{\bX}$ that $\formal{R}{\bX}$ is a domain,
that is $ST=0$ implies $S=0$ or $T=0$. Thus, for any $S,U\in\formal{R}{\bX}$
with $U\ne0$
there can be at most one $T\in\formal{R}{\bX}$ such that
$S=TU$ and if there is such a $T$ we may write $T=SU^{-1}$ without
ambiguity, even if $U$ is not invertible.

\section{Generating functions} \label{sc: generating functions}
\noindent
In this section we will exploit the bijections $\phi$, $\phisilly$,
and $f$ from Theorems~\ref{th:Sbij}, \ref{th:Ssillybij}, and~\ref{th:Nbij}
in their full splendour, and extract enumerative results in the form of
noncommutative generating functions.

Let $\bx_1,\dotsc,\bx_k,\bt$ be noncommutative variables.
If $n\in\N$ and $A_1,\dotsc,A_k$ are disjoint subsets of $[n]$,
we write
$[\bx_1^{A_1}\dotsm\bx_k^{A_k};{\bt}]_{n}$
for the monomial of length $n$ where
the $i$th factor is $\bx_j$ if $i\in A_j$ and $\bt$ if
$i\notin A_1\cup\dotsb\cup A_k$.
For instance,
$\monomthree{\bx}{\by}{\bt}{5}{\{2,4\}}{\{1\}}=\by\bx\bt\bx\bt$.

Let us define three generating functions $\bS$, $\bSsilly$ and $\bN$, where
$\bS$ counts permutations $\pi$ with respect to the sets
$P(\pi)$, $Q(\pi)$, and $\adjAsc(\pi)$, $\bSsilly$ counts permutations $\pi$
with respect to $\Psillyplus(\pi)$ and $\Qsillyplus(\pi)$, and $\bN$ counts
matchings $M$
without left nestings with respect to $\Rne(M)$, $\Rcrproper(M)$,
and $\LRcr(M)$.
\begin{align*}
\bS(\bx,\by,\bz,\bt,s) &:=
\sum_{n=1}^\infty \sum_{\pi\in\calS_n}
s^{\rmin(\pi)}\monomfour{\bx}{\by}{\bz}{\bt}{n}{P(\pi)}{Q(\pi)}{\adjAsc(\pi)}, \\
\bSsilly(\bx,\by,\bt,s) &:=
\sum_{n=1}^\infty \sum_{\pi\in\calS_n}
s^{\rmin(\pi)}
\monomthree{\bx}{\by}{\bt}{n}{\Psillyplus(\pi)}{\Qsillyplus(\pi)}, \\
\bN(\bx,\by,\bz,\bt,s) &:=
\sum_{n=1}^\infty \sum_{M\in\calN_n}
s^{\min(M)}\monomfour{\bx}{\by}{\bz}{\bt}{n}{\Rne(M)}
{\Rcrproper(M)}{\LRcr(M)}.
\end{align*}

Our goal in this section
is to show that $\bS(\bx,\by,\bz,\bt,1)=\bN(\bx,\by,\bz,\bt,1)$
and $\bS(\bx,\by,\by,\bt,s)=\bSsilly(\bx,\by,\bt,s)=\bN(\bx,\by,\by,\bt,s)$
and to obtain expressions for these functions.
To this end we will make extensive use
of the sieve principle,
which takes the following form in terms of noncommutative
generating functions.

\begin{theo}[The sieve principle]
Let $n$ be a positive integer and let
$A_1,\dotsc,A_k$ be disjoint subsets of $[n]$.
Then,
\[
[(\bu_1+\bt)^{A_1}\dotsm(\bu_k+\bt)^{A_k};{\bt}]_n
=\sum_{Y\subseteq A_1\cup\dotsb\cup A_k}
[\bu_1^{Y\cap A_1}\dotsm\bu_k^{Y\cap A_k};\bt]_n.
\]
\end{theo}
\begin{proof}
The left-hand side
is a product of factors, some of which are $\bt$ and some of which are
$(\bu_j+\bt)$ for some $j$. Expanding all factors of type $(\bu_j+\bt)$
yields the right-hand side.
\end{proof}

It turns out that it is convenient to make a variable substitution and
work instead with the generating functions
\begin{align}
\label{eq:Ssubst}
S(\bu,\bv,\bw,\bt,s) &:= \bS(\bu+\bv+\bt,\bv+\bt,\bv+\bw+\bt,\bt,s), \\
\label{eq:Ssillysubst}
\Ssilly(\bu,\bv,\bt,s) &:= \bSsilly(\bu+\bv+\bt,\bv+\bt,\bt,s), \\
\label{eq:Nsubst}
N(\bu,\bv,\bw,\bt,s) &:= \bN(\bu+\bv+\bt,\bv+\bt,\bv+\bw+\bt,\bt,s).
\end{align}
Of course, we can reverse the substitution by the formulas
$\bS(\bx,\by,\bz,\bt,s)=S(\bx-\by,\by-\bt,\bz-\by,\bt,s)$,
$\bSsilly(\bx,\by,\bt,s)=\Ssilly(\bx-\by,\by-\bt,\bt,s)$, and
$\bN(\bx,\by,\bz,\bt,s)=N(\bx-\by,\by-\bt,\bz-\by,\bt,s)$ if we want.

\begin{prop}\label{pr:sieveSN}
The following identities hold.
\begin{align*}
S(\bu,\bv,\bw,\bt,s) &=
\sum_{\bpi \in \bcalS}
s^{\rmin(\bpi)}
\monomfour{(\bu+\bv)}{\bv}{(\bv+\bw)}{\bt}{n(\bpi)}%
{\barred(\bpi)\cap P(\bpi)}{\barred(\bpi)\cap Q(\bpi)}%
{\barred(\bpi)\cap R(\bpi)},\\
\Ssilly(\bu,\bv,\bt,s) &=
\sum_{\hpi \in \hcalS}
s^{\rmin(\hpi)}
\monomthree{(\bu+\bv)}{\bv}{\bt}{n(\hpi)}%
{\hattedplus(\hpi)\cap \Psillyplus(\hpi)}%
{\hattedplus(\hpi)\cap \Qsillyplus(\hpi)},\\
N(\bu,\bv,\bw,\bt,s) &=
\sum_{\bM \in \bcalN}
s^{\min(\bM)}
\monomfour{(\bu+\bv)}{\bv}{(\bv+\bw)}{\bt}{n(\bM)}%
{\markedadj(\bM)\cap \Rne(\bM)}%
{\markedadj(\bM)\cap \Rcrproper(\bM)}%
{\markedadj(\bM)\cap \LRcr(\bM)}.
\end{align*}
\end{prop}
\begin{proof}
We will prove only the first identity since the proofs of the other two
are completely analogous.

By the sieve principle,
\begin{multline*}
\monomfour{(\bu+\bv+\bt)}{(\bv+\bt)}{(\bv+\bw+\bt)}{\bt}{n(\pi)}%
{P(\pi)}{Q(\pi)}{\adjAsc(\pi)} =  \\
\sum_{X \subseteq P(\pi)\cup Q(\pi)\cup \adjAsc(\pi)}
\monomfour{(\bu+\bv)}{\bv}{(\bv+\bw)}{\bt}{n(\pi)}{X\cap P(\pi)}%
{X\cap Q(\pi)}{X\cap \adjAsc(\pi)}
\end{multline*}
and thus, since $P(\pi)\cup Q(\pi)\cup \adjAsc(\pi)=\Asc(\pi)$, we have
\begin{align*}
S(\bu,\bv,\bw,\bt, s)
&= \sum_{\pi \in \calS}s^{\rmin(\pi)}
\sum_{X \subseteq \Asc(\pi)}
\monomfour{(\bu+\bv)}{\bv}{(\bv+\bw)}{\bt}{n(\pi)}%
{X\cap P(\pi)}{X\cap Q(\pi)}{X\cap \adjAsc(\pi)} \\
&=\sum_{\bpi \in \bcalS}s^{\rmin(\bpi)}
\monomfour{(\bu+\bv)}{\bv}{(\bv+\bw)}{\bt}{n(\bpi)}%
{\barred(\bpi)\cap P(\bpi)}{\barred(\bpi)\cap Q(\bpi)}%
{\barred(\bpi)\cap R(\bpi)}.
\end{align*}
\end{proof}

For $\mu\in\{\min,\max\}$,
we define the generating functions
\[
I^\mu(\bu,\bv,\bw,\bt, s)
= \sum_{T \in \flatcolstrictfillings}s^{\mu(T)}
\monomfour{(\bu+\bv)}{\bv}{(\bv+\bw)}{\bt}{n(T)}%
{X(T)\cap \Des(T)}{X(T)\cap \Asc(T)}%
{X(T)\cap \Rep(T)}.
\]
In the light of Proposition~\ref{pr:sieveSN}, Theorems~\ref{th:Sbij},
\ref{th:Ssillybij}, and~\ref{th:Nbij}
immediately yield the following theorem.
\begin{theo}\label{th:SIN}
The identities
\begin{align*}
S(\bu,\bv,\bw,\bt,s) &= I^{\max}(\bu,\bv,\bw,\bt,s) \\
\Ssilly(\bu,\bv,\bt,s) &= I^{\max}(\bu,\bv,0,\bt,s) \\
N(\bu,\bv,\bw,\bt,s) &= I^{\min}(\bu,\bv,\bw,\bt,s)
\end{align*}
hold.
\end{theo}

What remains now is to show the equalities
$I^{\max}(\bu,\bv,\bw,\bt,1)=I^{\min}(\bu,\bv,\bw,\bt,1)$
and
$I^{\max}(\bu,\bv,0,\bt,s)=I^{\min}(\bu,\bv,0,\bt,s)$
and to derive expressions for these functions.

\begin{lemma}\label{lm:defF}
For $\mu\in\{\min,\max\}$, we have
\begin{equation}\label{eq:step}
I^{\mu}(\bu,\bv,\bw,\bt,s)
= \sum_{m=1}^\infty \prod_{k=1}^m
\bigl(\bt\bv^{-1}F^{\mu}_k(\bu,\bv,\bw,s)\bigr),
\end{equation}
where
\[
F^{\mu}_k(\bu,\bv,\bw,s)
=\sum_T
s^{\mu(T)}\monomthree{(\bu+\bv)}{(\bv+\bw)}{\bv}{\ell(T)}{\Des(T)}{\Rep(T)},
\]
the
sum taken over all column-strict fillings $T$ of rectangular shapes
of height $k$.
\end{lemma}
\begin{proof}
A flat shape can be seen as a finite sequence of rectangular \emph{blocks},
where all columns of height $k$ constitute the $k$th block.
An alternative way of describing a flat column-strict filling is to specify
the number $m$ of blocks and then for each $k\in[m]$ specify
the $k$th block, that is,
a column-strict filling of a rectangular shape of height $k$.
But this is exactly what the right-hand side of
\eqref{eq:step} does. Furthermore, in both the left-hand side
and the right-hand side
\begin{itemize}
\item
$\bt$ keeps track of the leftmost dot in each block,
\item
$(\bu+\bv)$ keeps track of the descents $\alpha_{i-1}>\alpha_i$
lying inside a block (i.e.~such that columns $i-1$ and
$i$ have the same height),
\item
$\bv$ keeps track of the ascents $\alpha_{i-1}<\alpha_i$
lying inside a block,
\item
$(\bv+\bw)$ keeps track of the repetitions $\alpha_{i-1}=\alpha_i$
lying inside a block, and finally
\item
$s$ counts the minima or maxima of $T$.
\end{itemize}
\end{proof}

Now, we are just one final step from success.
\begin{lemma}\label{lm:F}
The generating functions $F^{\mu}_k(\bu,\bv,\bw,s)$ as defined in
Lemma~\ref{lm:defF} are given by
\begin{align*}
F^{\max}_k(\bu, \bv, \bw, s)
&=\bigl(1-\bv\bu^{-1}(AB-1)\bigr)^{-1}-1, \\
F^{\min}_k(\bu, \bv, \bw, s)
&=\bigl(1-\bv\bu^{-1}(BA-1)\bigr)^{-1}-1,
\end{align*}
where
\begin{align*}
A = 1+s\bu(1-s\bw)^{-1}, \\
B = (1+\bu(1-\bw)^{-1})^{k-1}.
\end{align*}
In particular, we have
\begin{align*}
F^{\mu}_k(\bu, \bv, \bw, 1)
&= \bigl(1-\bv\bu^{-1}[ ( 1+\bu[1-\bw]^{-1})^k - 1]\bigr)^{-1}-1, \\
F^{\mu}_k(\bu, \bv, 0, s)
&= \bigl(1-\bv\bu^{-1}[ ( 1+\bu)^{k-1}(1 + s\bu) - 1]\bigr)^{-1}-1
\end{align*}
for $\mu\in\{\min,\max\}$.
\end{lemma}

\begin{proof}
By the sieve principle we have
\begin{equation}\label{eq:Ftemp}
F^{\mu}_k(\bu,\bv,\bw,s)
=\sum_T s^{\mu(T)}\sum_{Y\subseteq\Des(T)\cup\Rep(T)}
\monomthree{\bu}{\bw}{\bv}{\ell(T)}{Y\cap\Des(T)}{Y\cap\Rep(T)}.
\end{equation}
This generating function chooses a
column-strict filling $T$ of a rectangular shape of height $k$,
and a set $Y\subseteq\Des(T)\cup\Rep(T)$.
Divide the shape $\lambda(T)$ into rectangular blocks of height $k$,
so that adjacent columns $i-1$ and $i$ belong to the same block if
and only if $i\in Y$. Since the dots inside each block form a weakly descending
chain (i.e.~each dot is below or at the same level as the dot in the column
immediately to the left),
we get
\[
F^{\mu}_k(\bu,\bv,\bw,s) = (1-\bv\bu^{-1}G^{\mu}_k(\bu,\bw,s))^{-1}-1,
\]
where
\[
G^{\mu}_k(\bu,\bw,s)=
\sum_T s^{\mu(T)}\monomtwo{\bw}{\bu}{\ell(T)}{\Rep(T)},
\]
the sum taken over column-strict fillings $T$ of rectangular shapes of
height $k$ such that $\Asc(T)=\emptyset$.
The generating function $G^{\mu}_k(\bu,\bw,s)$ only has to choose the
number of dots on each row, so $G^{\max}_k(\bu,\bw,s)=AB-1$ and
$G^{\min}_k(\bu,\bw,s)=BA-1$.
\end{proof}

Finally, we are ready to present our main theorem.
\begin{theo}\label{th:noncommutativemain}
The following equalities of generating functions hold:
\begin{multline*}
S(\bu,\bv,\bw,\bt,1)=N(\bu,\bv,\bw,\bt,1)=\\
\sum_{m=1}^\infty \prod_{k=1}^m
\bt\bv^{-1}
\bigl[\bigl(1-\bv\bu^{-1}[(1+\bu[1-\bw]^{-1})^k-1]\bigr)^{-1}-1\bigr]
\end{multline*}
and
\begin{multline*}
S(\bu,\bv,0,\bt,s)=\Ssilly(\bu,\bv,0,\bt,s)=N(\bu,\bv,0,\bt,s)=\\
\sum_{m=1}^\infty \prod_{k=1}^m
\bt\bv^{-1}
\bigl[\bigl(1-\bv\bu^{-1}[(1+\bu)^{k-1}(1+s\bu)-1]\bigr)^{-1}-1\bigr].
\end{multline*}
This implies, via~\eqref{eq:Ssubst}, \eqref{eq:Ssillysubst},
and \eqref{eq:Nsubst}, that
$\bS(\bx,\by,\bz,\bt,1)=\bN(\bx,\by,\bz,\bt,1)$ and
$\bS(\bx,\by,\by,\bt,s)=\bSsilly(\bx,\by,\by,\bt,s)=\bN(\bx,\by,\by,\bt,s)$.
\end{theo}
\begin{proof}
The theorem follows from Theorem~\ref{th:SIN}
and Lemma~\ref{lm:defF} and~\ref{lm:F}.
\end{proof}

\section{Abelianization}\label{sc: abel}
\noindent
In this section we will replace the noncommutative variables in the
generating functions from the previous section by commutative variables
and relate the resulting expressions to known enumerative results.
At the end we will obtain a proof of Conjecture~21 in~\cite{ClaLin2011}.

From~\eqref{eq:Ssubst} it is clear that
\[
\bS(\bx\bt,\by\bt,\bz\bt,\bt,s)
=S((\bx-\by)\bt,(\by-1)\bt,(\bz-\by)\bt,\bt,s).
\]
Substituting commutative variables $x,y,z,t$ for $\bx,\by,\bz,\bt$ and using
Theorem~\ref{th:noncommutativemain} yields the following theorem.
\begin{theo}\label{th:main}
The following identities of generating functions hold:
\begin{multline*}
\sum_{n=1}^\infty t^n \sum_{\pi\in\calS_n}
x^{p(\pi)}y^{q(\pi)}z^{\adjasc(\pi)}
=\sum_{n=1}^\infty t^n \sum_{M\in\calN_n}
x^{\rne(M)}y^{\rcrproper(M)}z^{\lrcr(M)} \\
=\sum_{m=1}^\infty\prod_{k=1}^m \frac{1}{y-1}
\left[
\left(1-\frac{y-1}{x-y}%
\left[\left(1+\frac{(x-y)t}{1-(z-y)t}\right)^k-1\right]\right)^{-1}-1\right]
\end{multline*}
and
\begin{align}
&\quad\sum_{n=1}^\infty t^n \sum_{\pi\in\calS_n}
s^{\rmin(\pi)}x^{p(\pi)}y^{q(\pi)+\adjasc(\pi)} \\
& =\sum_{n=1}^\infty t^n \sum_{\pi\in\calS_n}
s^{\rmin(\pi)}x^{\psilly(\pi)}y^{\qsilly(\pi)} \\
& =\sum_{n=1}^\infty t^n \sum_{M\in\calN_n}
s^{\min(M)}x^{\rne(M)}y^{\rcr(M)} \label{eq:conjtwentyone} \\
& =\sum_{m=1}^\infty\prod_{k=1}^m \tfrac{1}{y-1}
\left[\left(1-\tfrac{y-1}{x-y}%
\left[(1+(x-y)t)^{k-1}(1+(x-y)st)-1\right]\right)^{-1}-1\right].
\end{align}
\end{theo}

If we set $y=1$ we obtain the following corollary.
\begin{coro}
The following identitites of generating functions hold:
\begin{multline*}
\sum_{n=1}^\infty t^n \sum_{\pi\in\calS_n}
x^{p(\pi)}z^{\adjasc(\pi)}
=\sum_{n=1}^\infty t^n \sum_{M\in\calN_n}
x^{\rne(M)}z^{\lrcr(M)} \\
=\sum_{m=1}^\infty\prod_{k=1}^m \frac{1}{x-1}
\left[\left(1+\frac{(x-1)t}{1-(z-1)t}\right)^k-1\right],
\end{multline*}
and
\begin{align*}
& \sum_{n=1}^\infty t^n \sum_{\pi\in\calS_n}
s^{\rmin(\pi)}x^{p(\pi)} \\
= & \sum_{n=1}^\infty t^n \sum_{\pi\in\calS_n}
s^{\rmin(\pi)}x^{\psilly(\pi)} \\
= & \sum_{n=1}^\infty t^n \sum_{M\in\calN_n}
s^{\min(M)}x^{\rne(M)} \\
= & \sum_{m=1}^\infty\prod_{k=1}^m \frac{1}{x-1}
\left[(1+(x-1)t)^{k-1}(1+(x-1)st)-1\right].
\end{align*}
\end{coro}

Now, putting $z=1$ or $s=1$ in the equations of the last corollary yields
\[
\sum_{n=1}^\infty t^n \sum_{\pi\in\calS_n}
x^{p(\pi)}
=\sum_{n=1}^\infty t^n \sum_{M\in\calN_n}
x^{\rne(M)}
=\sum_{m=1}^\infty\prod_{k=1}^m \frac{1}{x-1}
\left[(1+(x-1)t)^k-1\right]
\]
and if we set $x=0$ into this expression we obtain Zagier's~\cite{zagier2001}
beautiful generating function for the Fishburn numbers:
\begin{multline*}
\sum_{n=1}^\infty t^n \#\{\pi\in\calS_n\colon p(\pi)=0\}
=\sum_{n=1}^\infty t^n \#\{M\in\calN_n\colon\rne(M)=0\} \\
=\sum_{m=1}^\infty\prod_{k=1}^m 
(1-(1-t)^k).
\end{multline*}

Conjecture~21 from \cite{ClaLin2011} can now be proved via the following
theorem.
\begin{theo}\label{th:conjtwentyone}
For any positive integer $n$ the following identity of generating functions
holds.
\[
\sum_{\pi\in\calS_n}s^{\rmin(\pi)}x^{\psilly(\pi)}w^{\des(\pi)}
=\sum_{M\in\calN_n}s^{\min(M)}x^{\rne(M)}w^{\inter(M)-1}
\]
\end{theo}
\begin{proof}
Note that $\des(\pi)=n-1-\asc(\pi)=n-1-\bigl(\psilly(\pi)+\qsilly(\pi)\bigr)$
for any permutation $\pi\in\calS_n$ and that
$\inter(M)-1=n-1-\radj(M)=n-1-\bigl(\rne(M)+\rcr(M)\bigr)$
for any $M\in\calN_n$. Now the theorem follows from~\eqref{eq:conjtwentyone}
in Theorem~\ref{th:main}.
\end{proof}
Seeing that Conjecture~21 from \cite{ClaLin2011} follows from the theorem
only requires two simple observations:
\begin{itemize}
\item
The pattern $p$ used by Claesson and Linusson is our pattern
$\psilly=$\patternPsillyplain rotated 180 degrees. Hence our $\rmin(\pi)$
corresponds to their $\lmax(\pi)$.
\item
The part of their conjecture stating that the triple $(\rne, \min, \lev-1)$
of statistics on posets is equidistributed with the triple
$(\rne, \min, \inter-1)$ on matchings follows from their
Proposition~17 together with the fact that
their bijection $h$ preserves $\rne$.
\end{itemize}

\section{Left crossings and open problems}\label{sc: open}

\noindent
So far, we have only counted right adjacencies. Turning to the left
adjacencies, left nestings are forbidden and thus easily counted. We
are, however, able to account for the number of left crossings as
well. 

\begin{theo}\label{th:leftcrossing}

The following identities of generating functions hold:
\begin{multline*}
\sum_{n=0}^\infty t^n \sum_{\pi\in\calS_n}
x^{p(\pi)}y^{q(\pi)}z^{\adjasc(\pi)}\upsilon^{\ascproper(\pi)}
=\sum_{n=0}^\infty t^n \sum_{M\in\calN_n}
x^{\rne(M)}y^{\rcrproper(M)}z^{\lrcr(M)}\upsilon^{\lcrproper(M)} \\
=\sum_{m=0}^\infty\prod_{k=1}^m \frac{1}{\upsilon y-1}
\left[
\left(1-\frac{\upsilon y-1}{\upsilon (x-y)}%
\left[\left(1+\frac{(x-y)\upsilon t}{1-(z-y\upsilon)t}\right)^k-1\right]\right)^{-1}-1\right]
\end{multline*}
and
\begin{align*}
&\quad\sum_{n=0}^\infty t^n \sum_{\pi\in\calS_n}
s^{\rmin(\pi)}x^{p(\pi)}y^{q(\pi)+\adjasc(\pi)}\upsilon^{\asc(\pi)} \\
& =\sum_{n=0}^\infty t^n \sum_{\pi\in\calS_n}
s^{\rmin(\pi)}x^{\psilly(\pi)}y^{\qsilly(\pi)}\upsilon^{\asc(\pi)} \\
& =\sum_{n=0}^\infty t^n \sum_{M\in\calN_n}
s^{\min(M)}x^{\rne(M)}y^{\rcr(M)} \upsilon^{\lcr(M)} \\
& =\sum_{m=0}^\infty\prod_{k=1}^m \tfrac{1}{\upsilon y-1}
\left[\left(1-\tfrac{\upsilon y-1}{\upsilon (x-y)}%
\left[(1+(x-y)\upsilon t)^{k-1}(1+(x-y)\upsilon st)-1\right]\right)^{-1}-1\right].
\end{align*}

\end{theo}

\begin{proof}
This theorem amounts only to noting that the number of left
adjacencies equals the number of right adjacencies in a matching,
since the number of switches from openers to closers is always one
more than the number of switches from closers to openers when we
traverse through a matching from left to right. Thus, with no left
nestings, the number of left crossings equals the number of right
crossings and nestings. Mapping long right adjacencies to \patternP
and \patternQ, the number of long left crossings must equal the
number of long ascents. Theorem \ref{th:main} then gives the first
generating function. For the second, the patterns \patternPsilly and
\patternQsilly refine silly ascents, which are equinumerous with
ordinary ascents.
\end{proof}

It would of course be interesting if we could also pinpoint the
positions of the left crossings. To this end, we need to adjust our
ascent definition somewhat. Let the \emph{ascent bottoms} be the set
$\Ascbottom(\pi) = \{\pi(j) : \pi(j) < \pi(j+1)\}$ and the
\emph{long ascent bottoms} be the set $\Ascbottomproper(\pi) =
\{\pi(j) : \pi(j) < \pi(j+1)-1\}$. 

Based on computations up to $n = 8$ and $n = 7$, respectively, we make
the following conjectures. 

\begin{conj}
The sets
\[
M(n, A, B, C) = \left\{ M \in \calN_n | \Lcr(M) = A, \Rne(M) =
B, \Rcr(M) = C \right\}
\]
and
\[
P(n, A, B, C) = \left\{ \pi \in \calS_n | \Ascbottom(\pi) = A,
\Psillyplus(\pi) = B, \Qsillyplus(\pi) = C \right\}
\]
are equinumerous for all $(n, A, B, C)$, where $n$ is a positive integer and
$A,B,C\subseteq[n]$.
\end{conj}

\begin{conj}
The sets
\[
M(n, A, B, C, D) = \left\{ M \in \calN_n | \Lcrproper(M) = A,
\LRcr(M) = B , \Rne(M) = C, \Rcrproper(M) = D \right\}
\]
and
\[
P(n, A, B, C, D) = \left\{ \pi \in \calS_n | \Ascbottomproper(\pi) =
A, \adjAsc(\pi) = B(\pi), P(\pi) = C, Q(\pi) = D \right\}
\]
are equinumerous for all $(n, A, B, C, D)$, where $n$ is a positive
integer and $A,B,C,D\subseteq[n]$.
\end{conj}

Letting go of right adjacencies, we are able to compute the noncommutative
generating function for left adjacencies, which constitutes further evidence
for the conjectures above.

\begin{theo}
The following equalities of generating functions hold.
\[
\begin{split}
& \sum_{n=1}^\infty \sum_{\pi\in\calS_n}
\monomthree{\bupsilon}{\bz}{\bt}{n}{\Ascbottomproper(\pi)}{\adjAsc(\pi)
- 1}
\\ 
&= \sum_{n=1}^\infty \sum_{M\in\calN_n}
\monomthree{\bupsilon}{\bz}{\bt}{n}{\Lcrproper(M)}{\LRcr(M) - 1} \\
&= \sum_{m=1}^\infty \prod_{k=1}^m (m + 1 - k) \Bigl[ 1 - (m + 1 - k)
  \bigl(1 - (\bz - \bupsilon)\bigr)^{-1} (\bupsilon - \bt) \Bigr]^{-1}
  \bigl(1 - (\bz - \bupsilon)\bigr)^{-1} \bt.
\end{split}
\]

\end{theo}

\begin{proof}
To establish the connection
between permutations and matchings, there is a fairly simple
bijection $f$ from matchings without left nestings
to permutations which maps left
crossings to ascents. We define the bijection recursively. 

Assume that $M$ is a matching of size $n$ without left nestings
and let $M'$ be the matching obtained by
removing arc $M_n$, i.e.~the arc with the rightmost closer.
We then define $\pi = f(M)$ from $\pi' = f(M')$ as
follows. Assume that $\opener(M_n) = \closer(M_j) - 1$:
\begin{itemize}
\item If $j = n$, insert $n$ to the far left of $\pi'$;
\item If $\opener(M_n) - 1$ is a closer, insert $n$ to the immediate
  left of $j$ in $\pi'$;
\item If $\opener(M_n) - 1 = \opener(M_k)$, insert $n$ to the
  immediate right of $k$ in $\pi'$.
\end{itemize}

It is not obvious that this yields a bijection, since the left of $j$
and the right of $k$ might coincide. However, it is not hard to see
that in the second case, $n$ is inserted to the left of an ascent top,
breaking an ascent. In the third case, $n$ is inserted to the right of
a descent top or to the far right, hence not breaking an 
ascent. Thus, the map is injective. It is also clear that left 
crossings are mapped to ascent bottoms, since the left crossing $k$ is
introduced in the third case, where $k$ becomes the ascent bottom of
$n$. It is also easy to check that double crossings are mapped to short
ascents.

Neither left crossings nor ascent bottoms can be destroyed by
adding more arcs or higher elements. 

Turning to the generating function, we note that ascent bottoms in a
permutation turn into ascent tops in the same permutation rotated 180
degrees. Ascent tops are counted by the patterns $P(\pi)$, $Q(\pi)$
and $\adjAsc(\pi)$. Thus, by Theorem~\ref{th:noncommutativemain} we have
\[
\begin{split}
& \sum_{n=1}^\infty \sum_{\pi\in\calS_n}
\monomthree{\bupsilon}{\bz}{\bt}{n}{\Ascproper(\pi)}{\adjAsc(\pi)} \\
&= \bS(\bupsilon, \bupsilon, \bz, \bt, 1) \\
&= S(0, \bupsilon - \bt, \bz-\bupsilon, \bt, 1) \\
&= \sum_{m = 1}^\infty \prod_{k = 1}^m \bt (\bupsilon - \bt)^{-1}
\Bigl( \bigl[1 - k (\bupsilon - \bt) (1-(\bz-\bupsilon))^{-1} \bigr]^{-1} - 1
\Bigr) \\
&= \sum_{m = 1}^\infty \prod_{k = 1}^m k \bt \bigl(1-(\bz-\bupsilon)\bigr)^{-1}
\bigl[1 - k (\bupsilon -
  \bt) \bigl(1-(\bz-\bupsilon)\bigr)^{-1} \bigr]^{-1}.
\end{split}
\]
Since rotating the permutation turns ascent top $k$ into ascent bottom
$n + 1 - k$, the generating function of ascent tops should be the
vertical reflection of the generating function for ascent
bottoms.

\end{proof}

To conclude the picture, we would like to allow left nestings. Is
there a nice bijection from all matchings of size $n$ to, for
instance, ternary trees? What kind of patterns would be the
counterparts of \patternP and \patternQ? What are the generating
functions? 

Another possibly fruitful path of generalisation could be to study
the distribution of left and right nestings and crossings in
\emph{partitions}.
Recent work of Chen et al.~\cite{CheFanZha2011} and of
Yan and Xu~\cite{YanXu2011} shows some progress in this direction.

\section{Acknowledgement}
\noindent
The second author was supported
by a grant from the Swedish Research Council (621-2009-6090).

\bibliographystyle{abbrv}
\bibliography{matchings}

\end{document}